\documentclass[12pt]{article}
\usepackage{amsthm,amsfonts, amsbsy, amssymb,amsmath,graphicx}
\usepackage{graphics}

\newtheorem{theorem}{Theorem}[section]

\newtheorem{corollary}{Corollary}[section]

 \def\thtext#1{
 \catcode`@=11
 \gdef\@thmcountersep{. #1}
 \catcode`@=12}

 \newcounter{Rk}[section]
 \renewcommand{\thtext}{\thesection.\arabic{Rk}}
 \newenvironment{remark}{\trivlist\item[\hskip\labelsep{\bf Remark}]
 \refstepcounter{Rk}{\bf\thesection.\arabic{Rk}.}}%
 {\endtrivlist}

 \newcounter{Df}[section]
 \renewcommand{\thtext}{\thesection.\arabic{Df}}
 \newenvironment{definition}{\trivlist\item[\hskip\labelsep{\bf Definition}]
 \par\refstepcounter{Df}{\bf\thesection.\arabic{Df}.}}%
 {\endtrivlist}

 \newenvironment{example}{\trivlist \item[\hskip\labelsep{\bf Example.}]}%
 {\endtrivlist}

\newcommand{\exam}{\raisebox{-0.35\height}{\includegraphics[width=1.0cm]{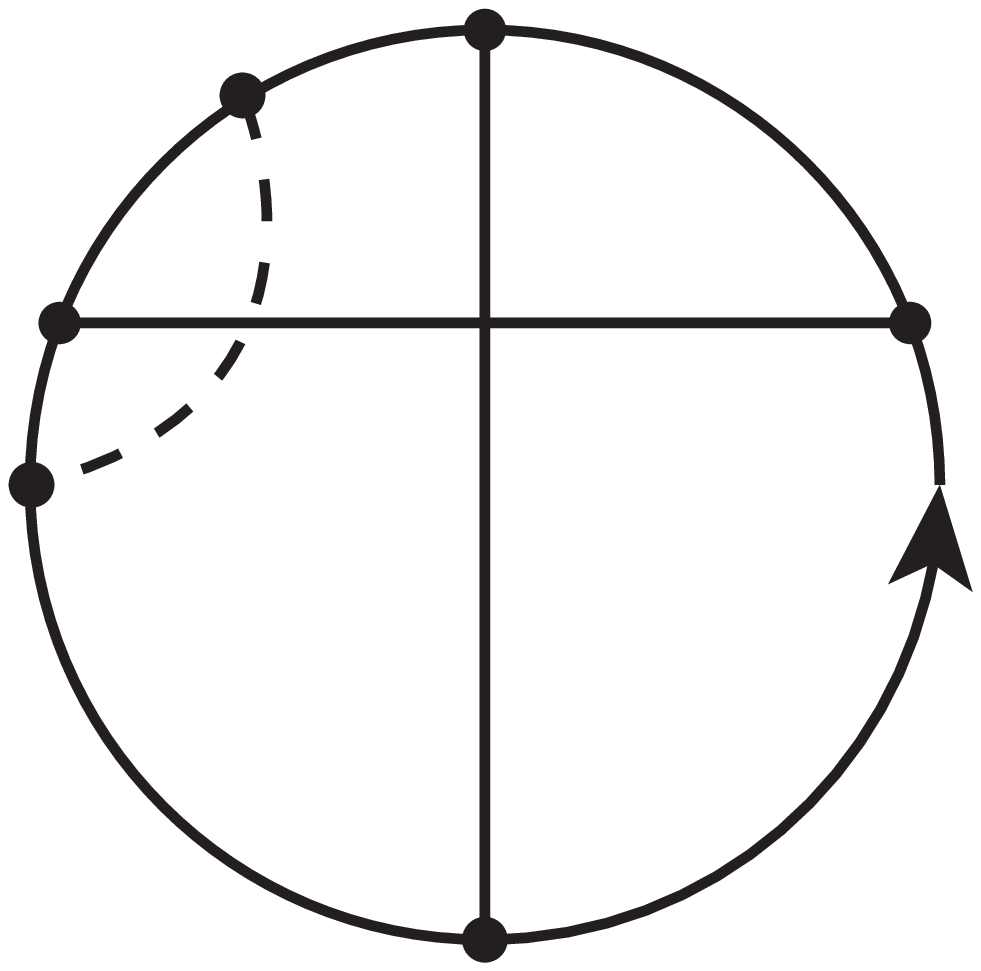}}}

\newcommand{\exone}{\raisebox{-0.44\height}{\includegraphics[width=1.2cm]{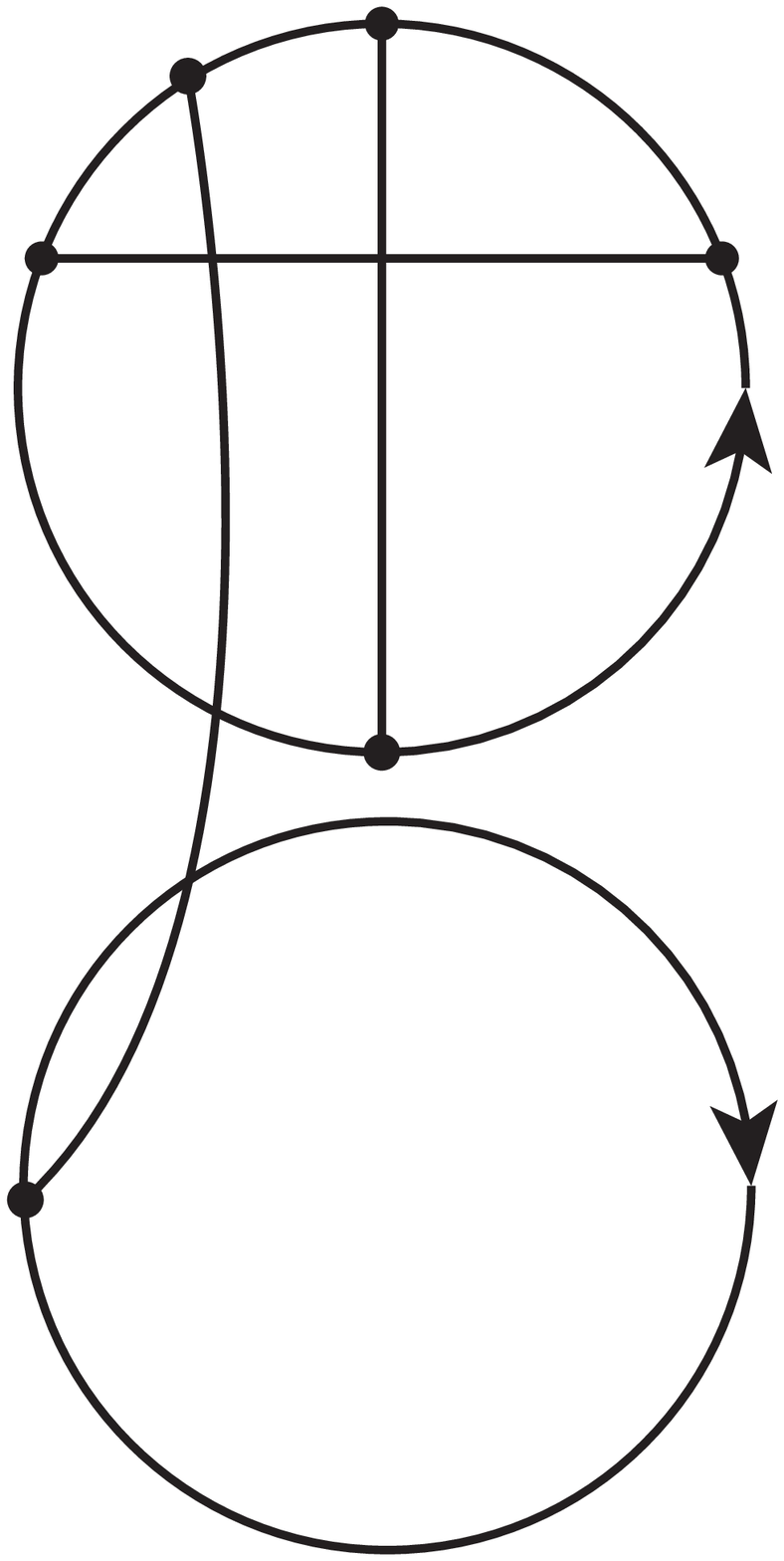}}}
\newcommand{\extwo}{\raisebox{-0.44\height}{\includegraphics[width=1.2cm]{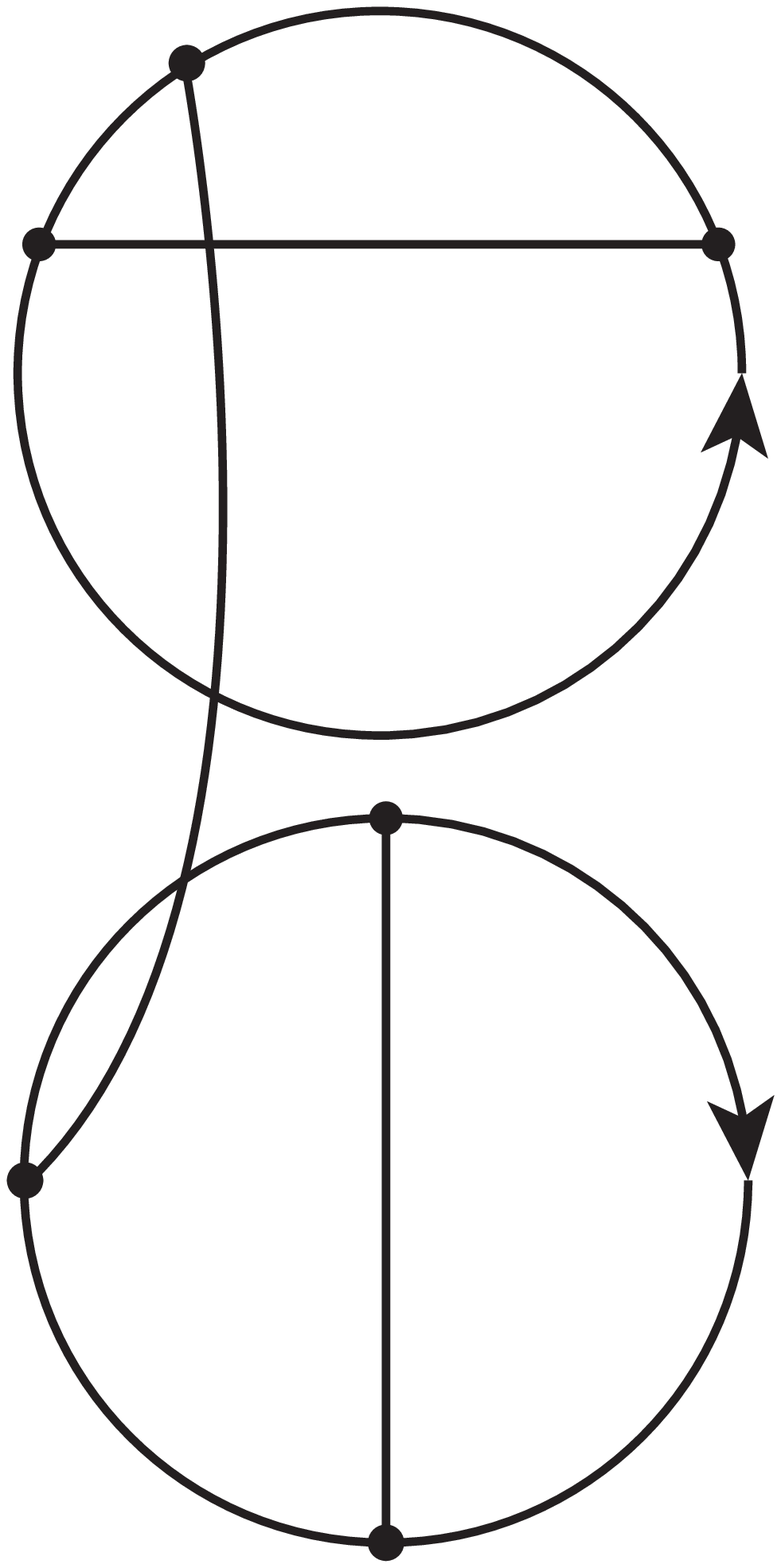}}}
\newcommand{\exthree}{\raisebox{-0.44\height}{\includegraphics[width=1.2cm]{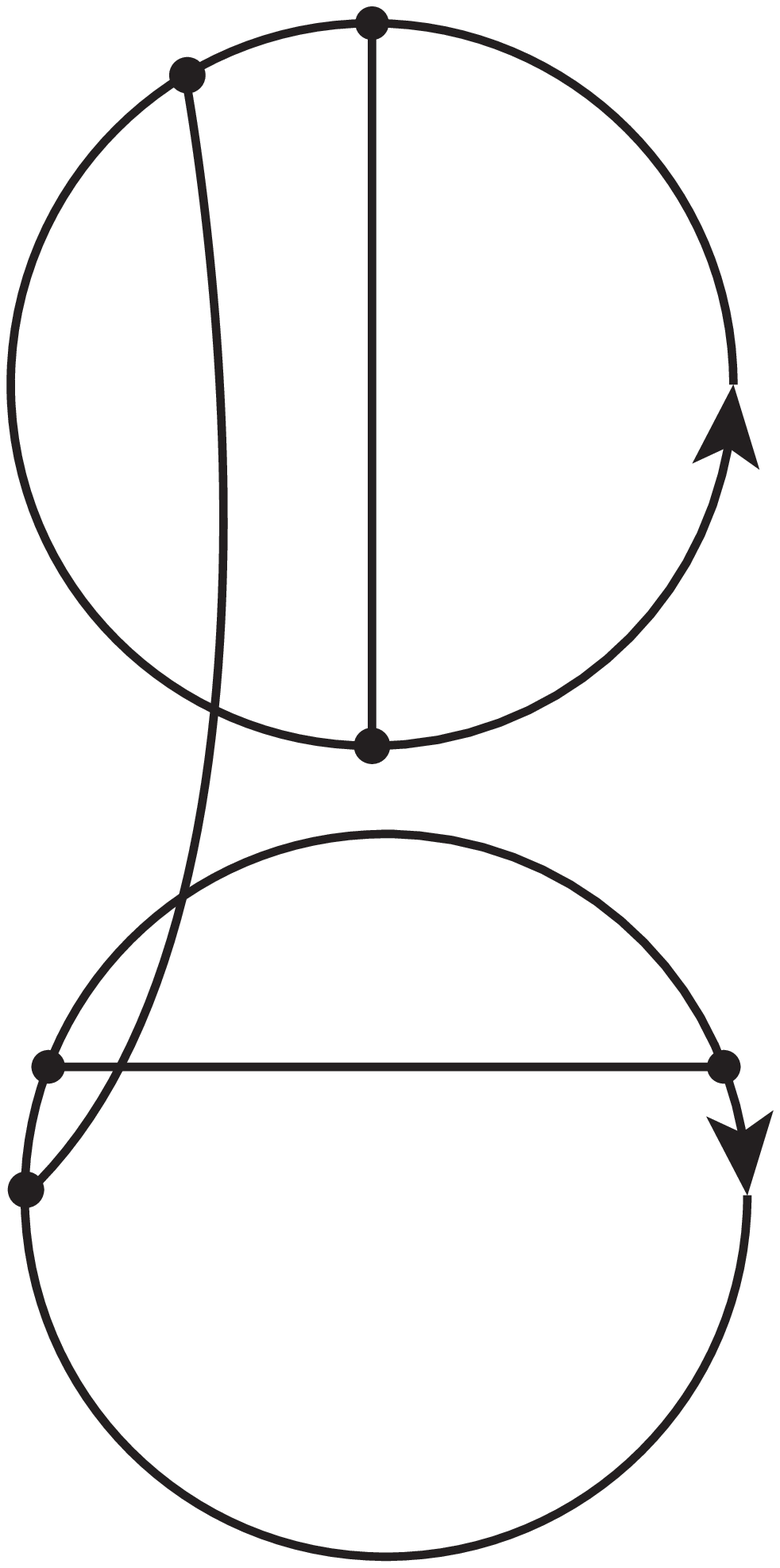}}}
\newcommand{\exfour}{\raisebox{-0.44\height}{\includegraphics[width=1.2cm]{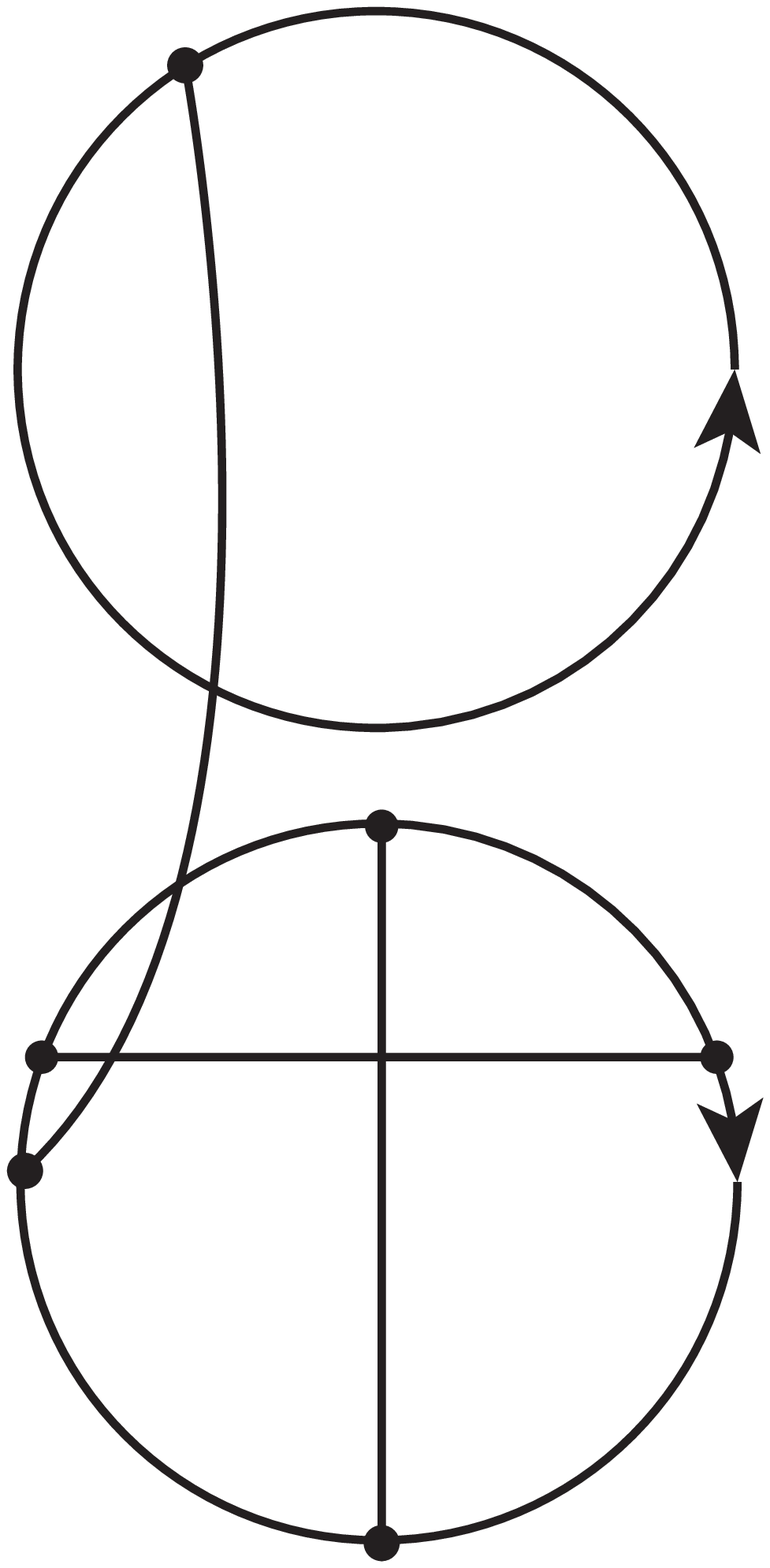}}}
\newcommand{\exfive}{\raisebox{-0.44\height}{\includegraphics[width=1.2cm]{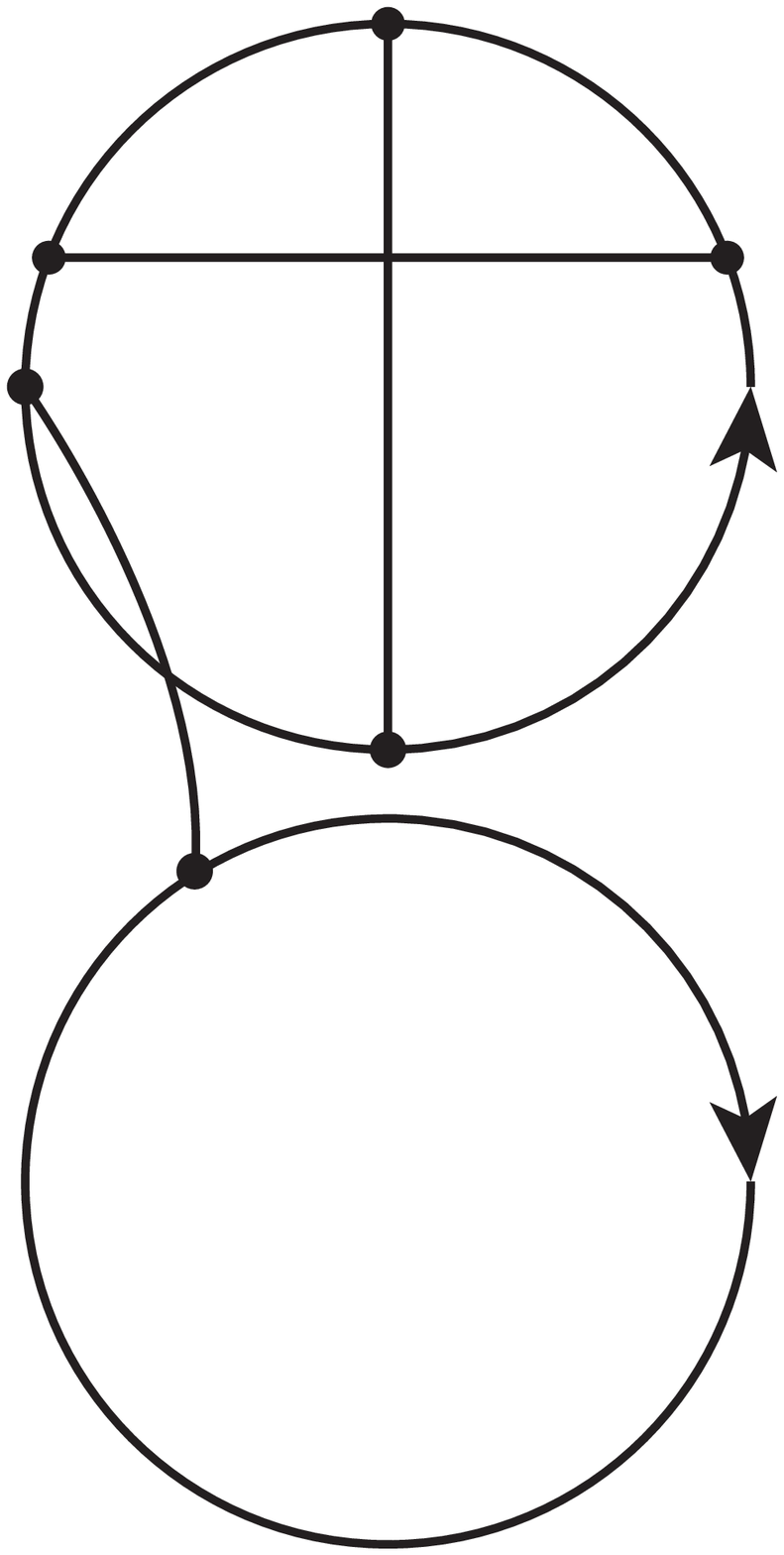}}}
\newcommand{\exsix}{\raisebox{-0.44\height}{\includegraphics[width=1.2cm]{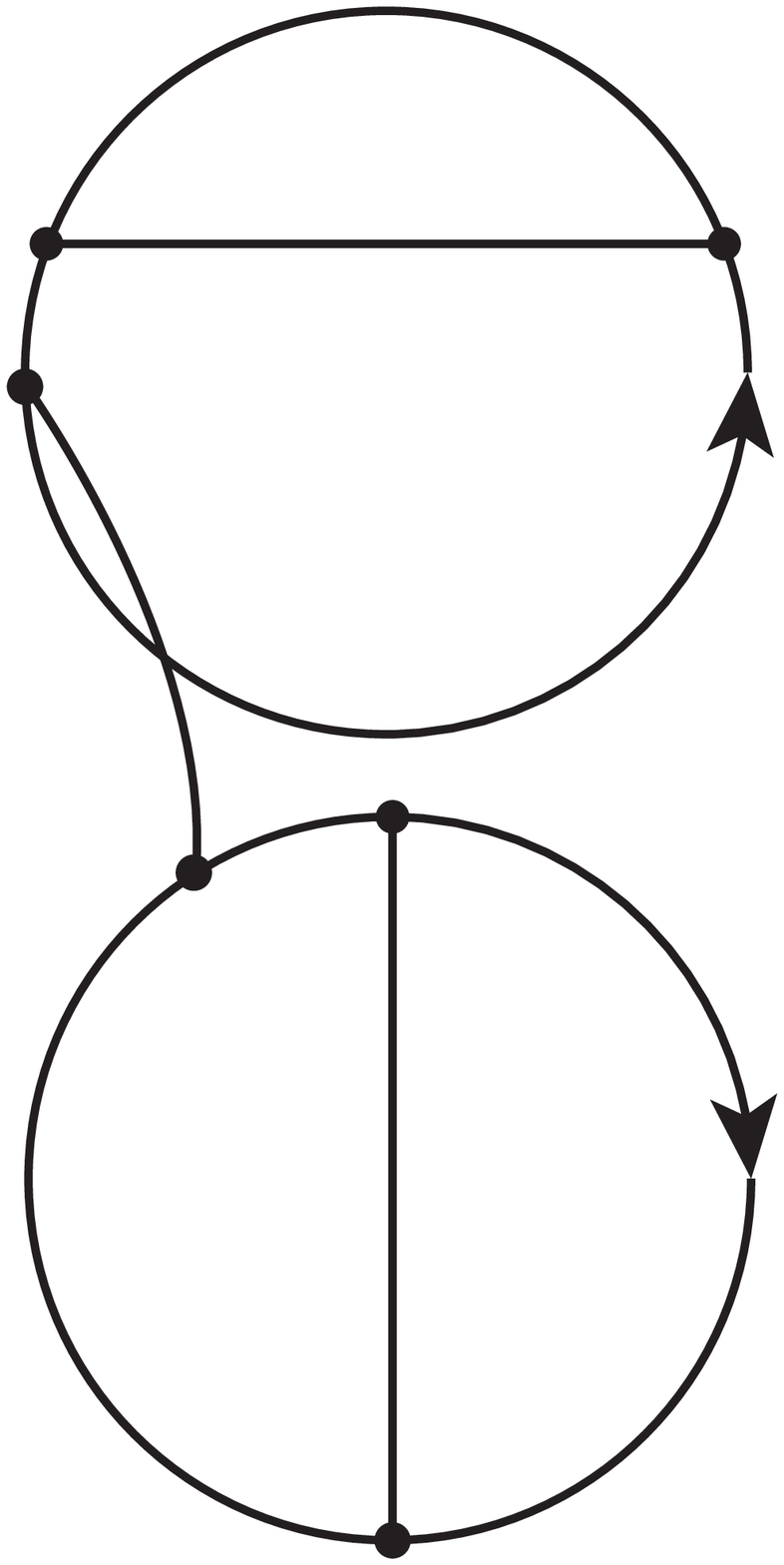}}}
\newcommand{\exseven}{\raisebox{-0.44\height}{\includegraphics[width=1.2cm]{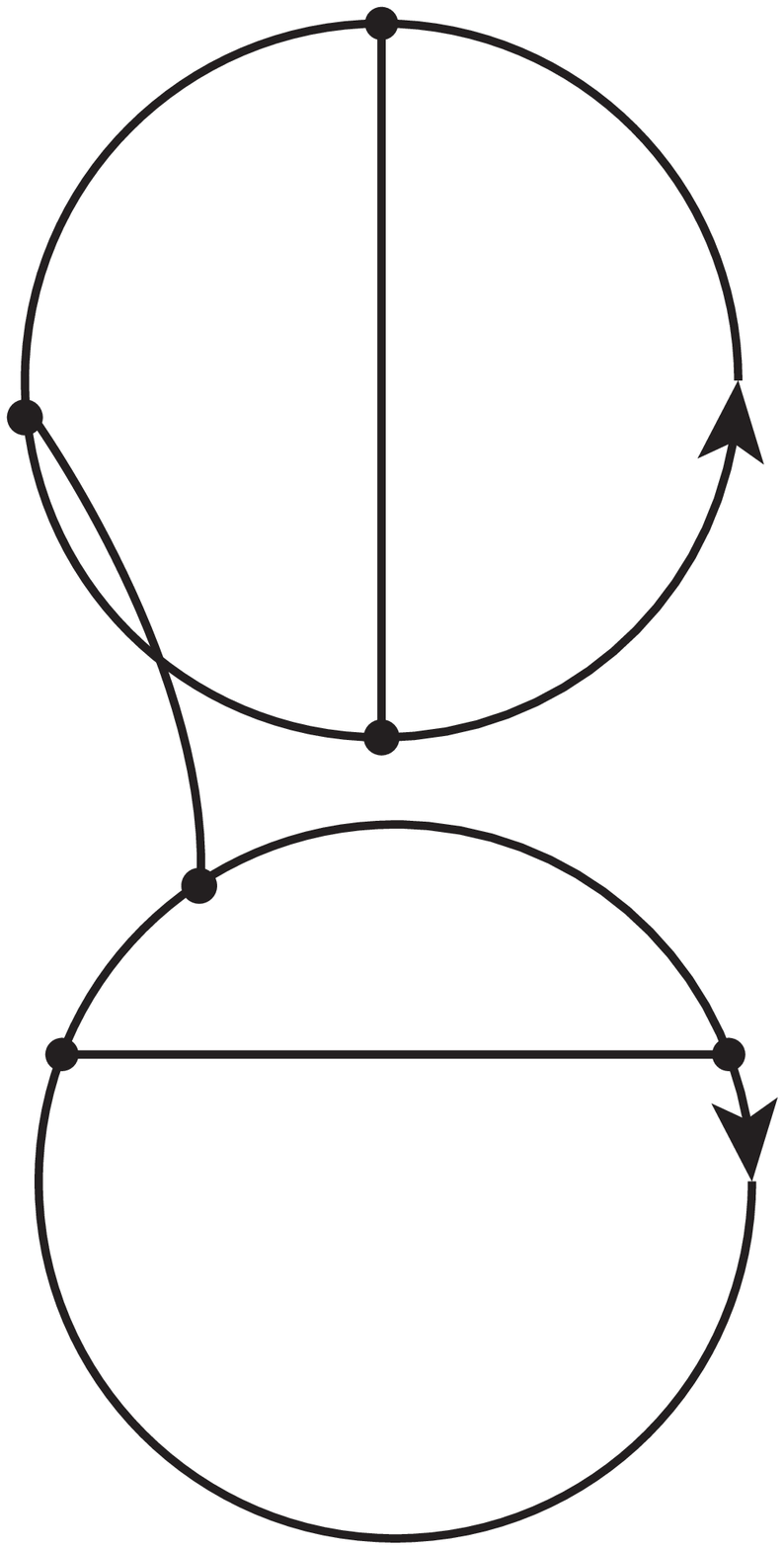}}}
\newcommand{\exeight}{\raisebox{-0.44\height}{\includegraphics[width=1.2cm]{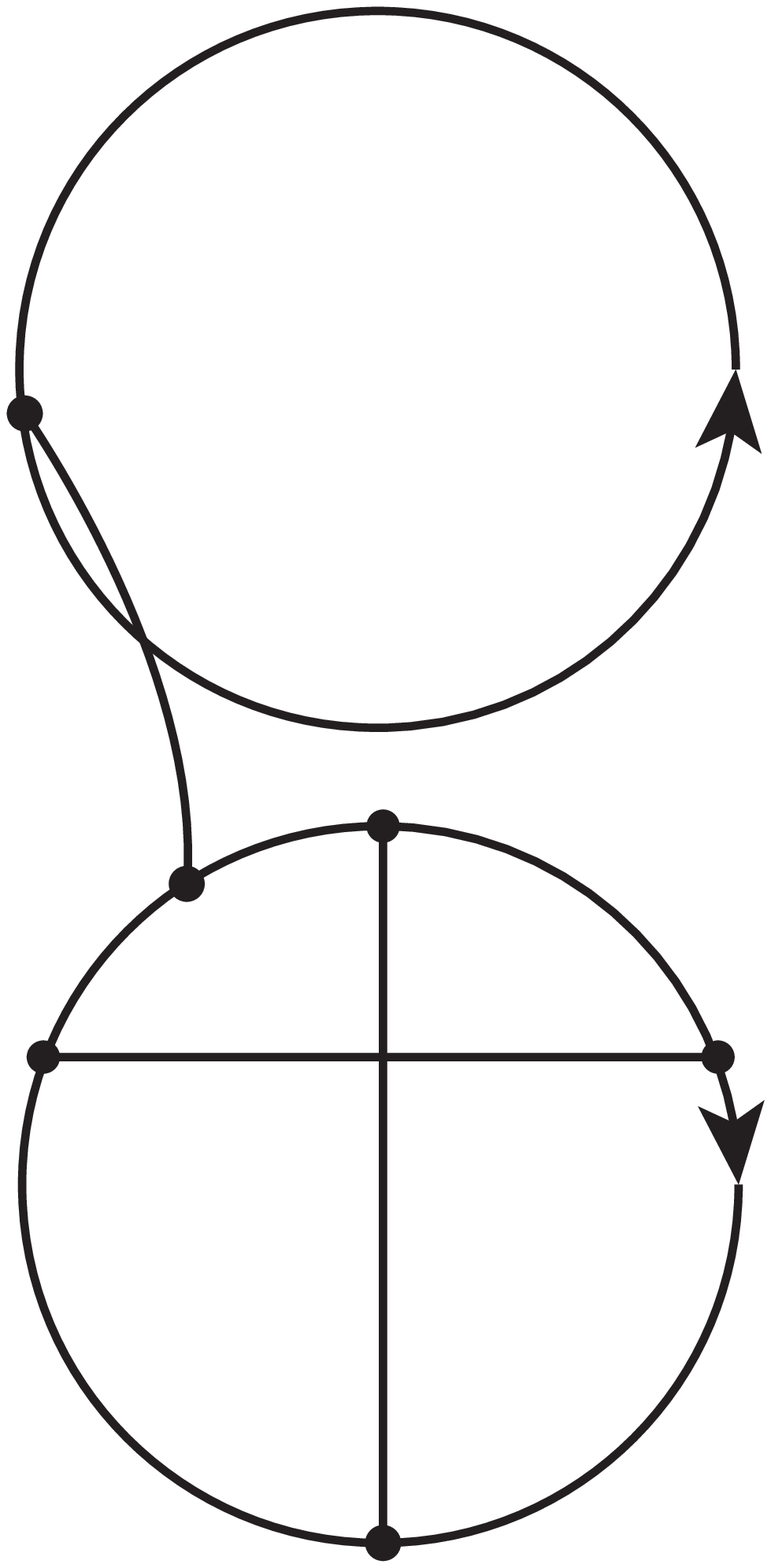}}}

\newcommand{\fl}{\raisebox{-0.35\height}{\includegraphics[width=1.0cm]{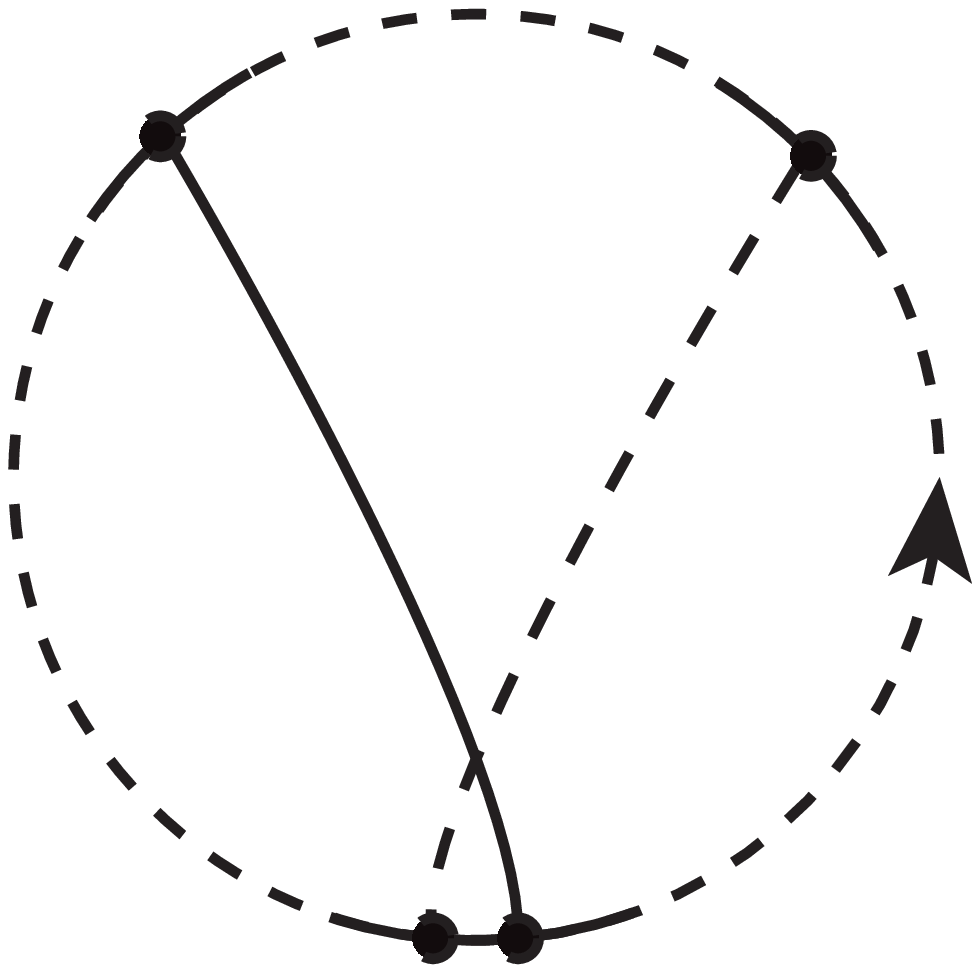}}}
\newcommand{\flo}{\raisebox{-0.44\height}{\includegraphics[width=1.2cm]{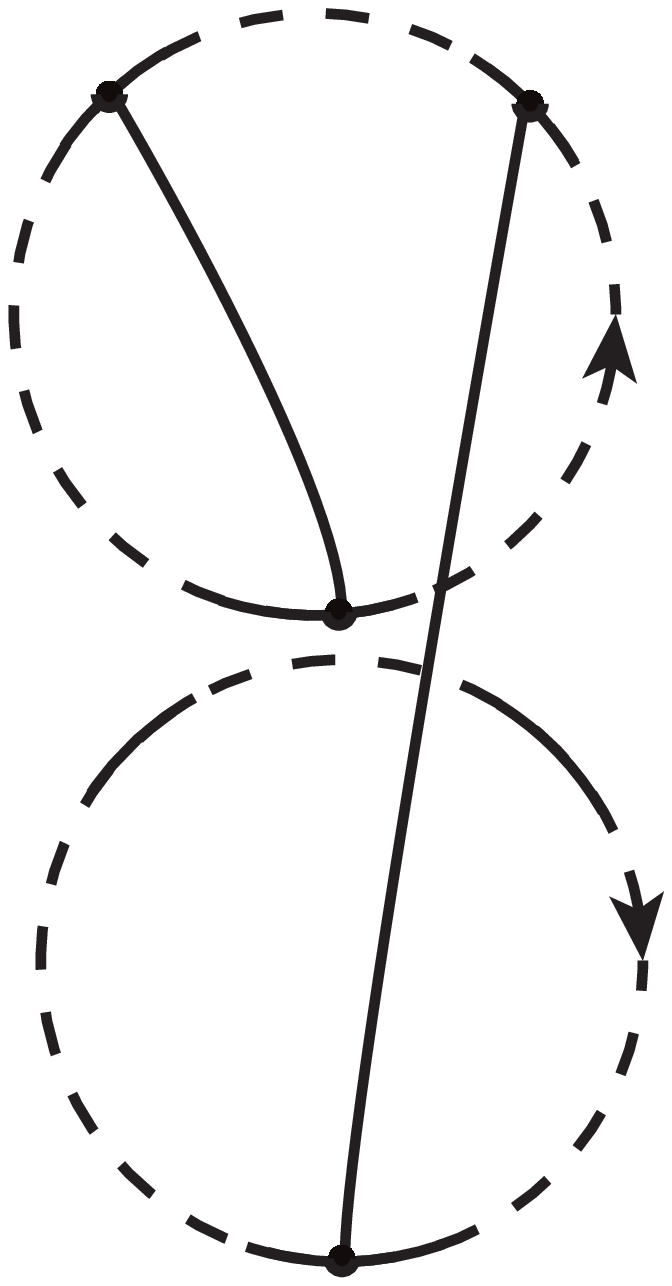}}}
\newcommand{\flt}{\raisebox{-0.44\height}{\includegraphics[width=1.2cm]{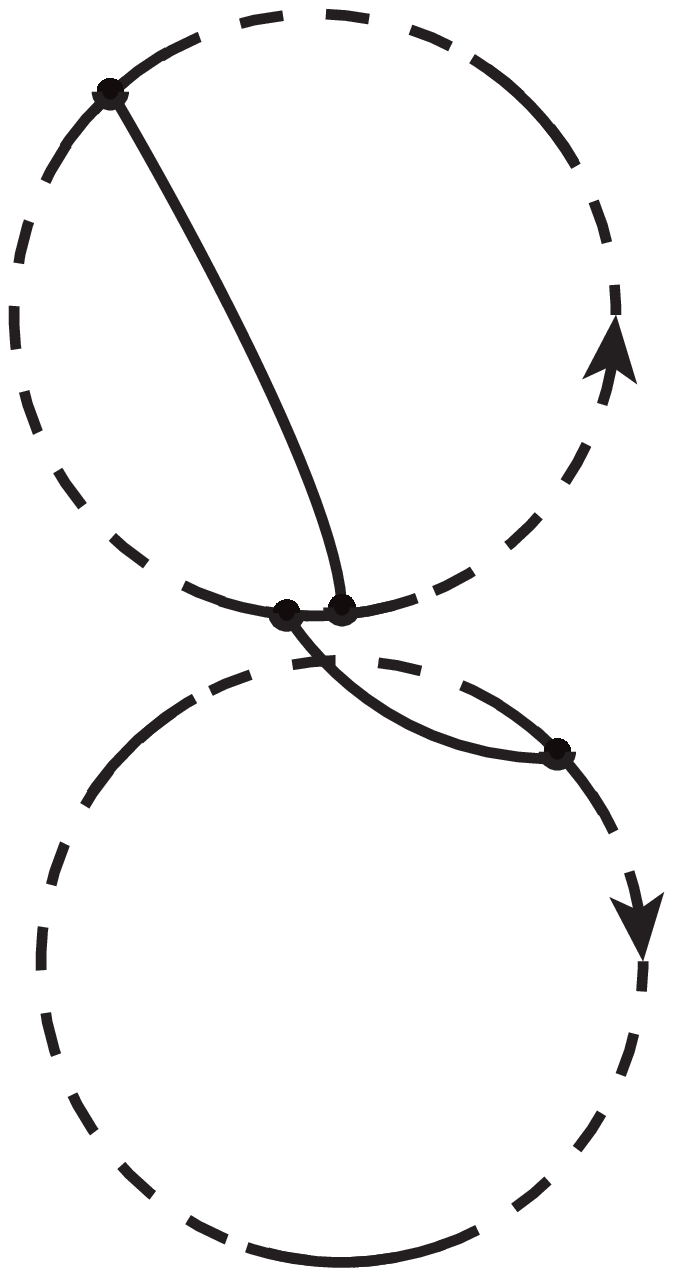}}}
\newcommand{\flth}{\raisebox{-0.44\height}{\includegraphics[width=1.2cm]{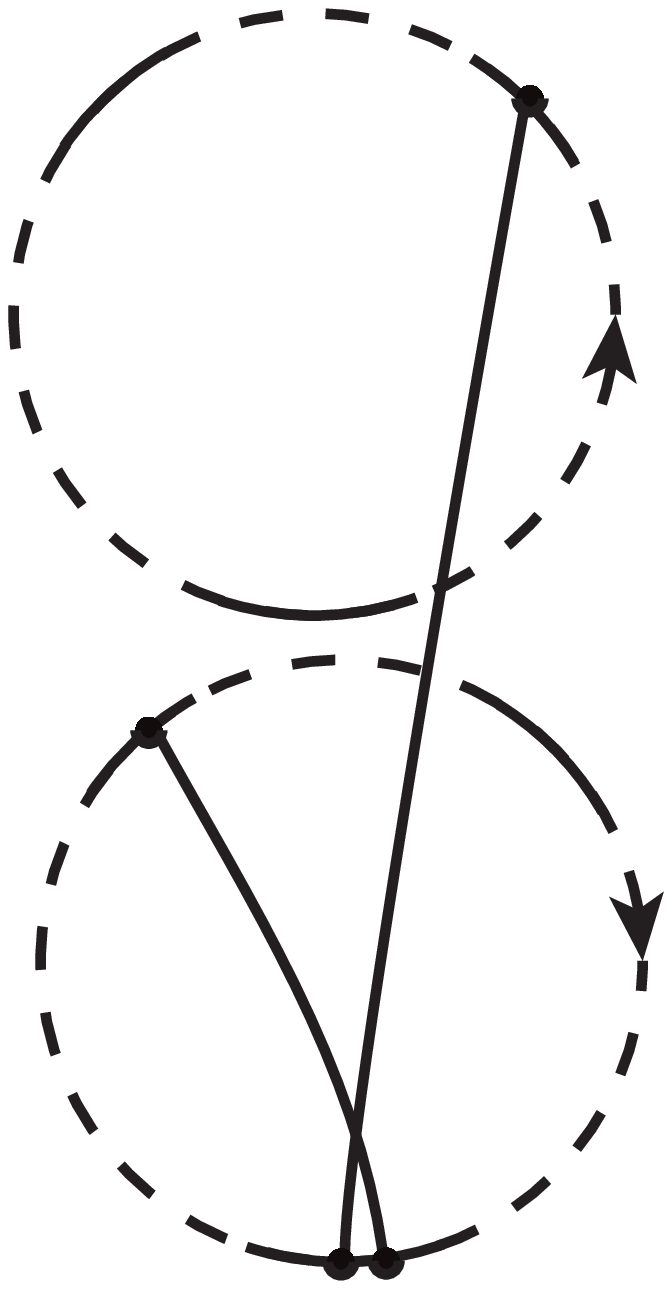}}}
\newcommand{\flf}{\raisebox{-0.44\height}{\includegraphics[width=1.2cm]{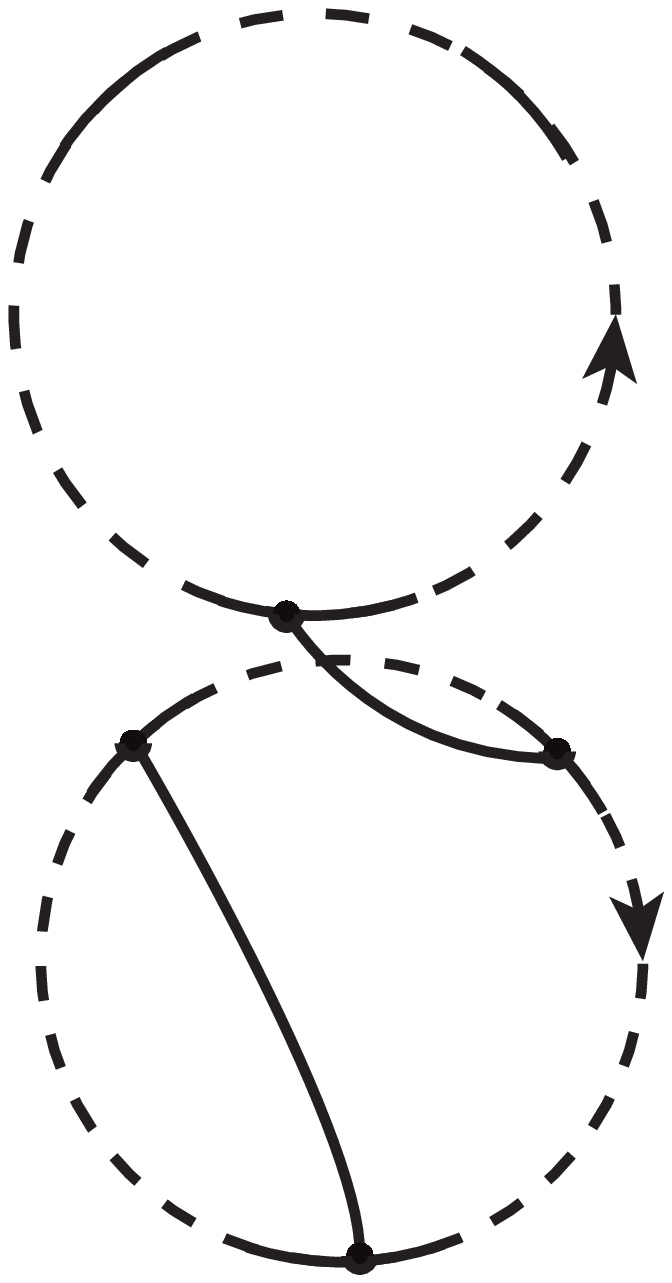}}}

\newcommand{\slef}{\raisebox{-0.35\height}{\includegraphics[width=1.0cm]{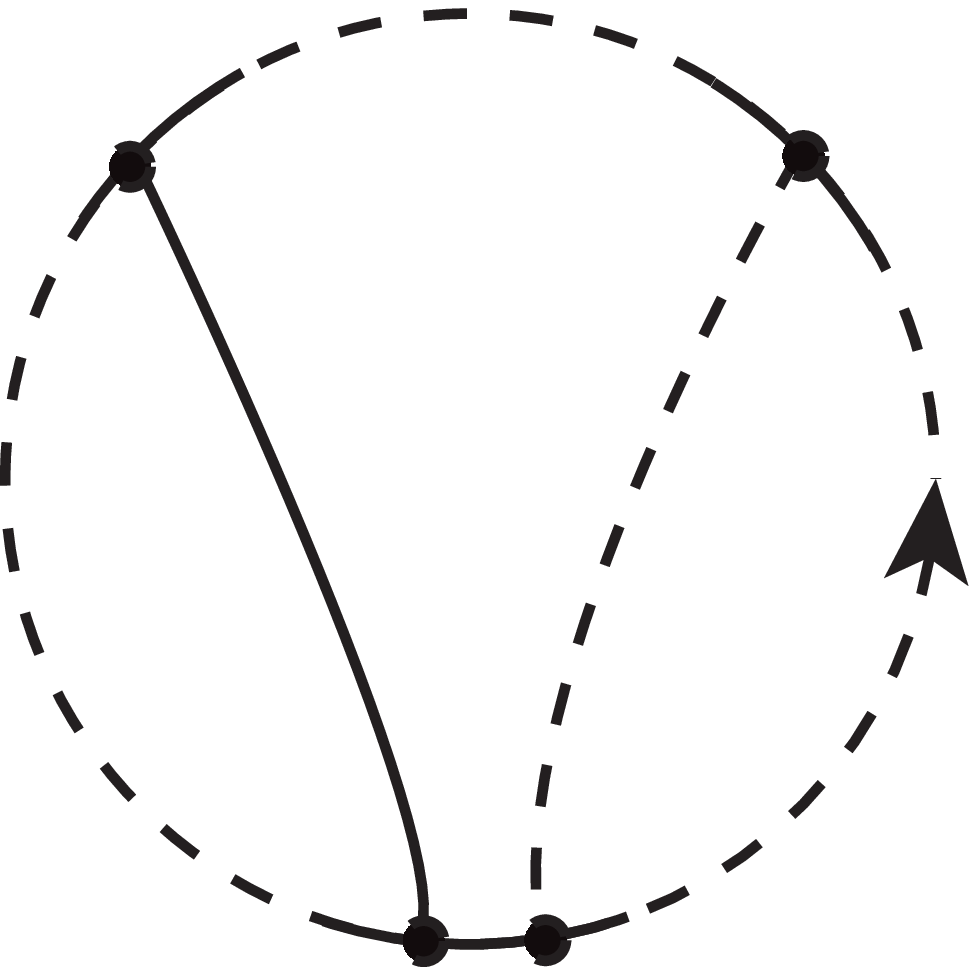}}}
\newcommand{\slo}{\raisebox{-0.44\height}{\includegraphics[width=1.2cm]{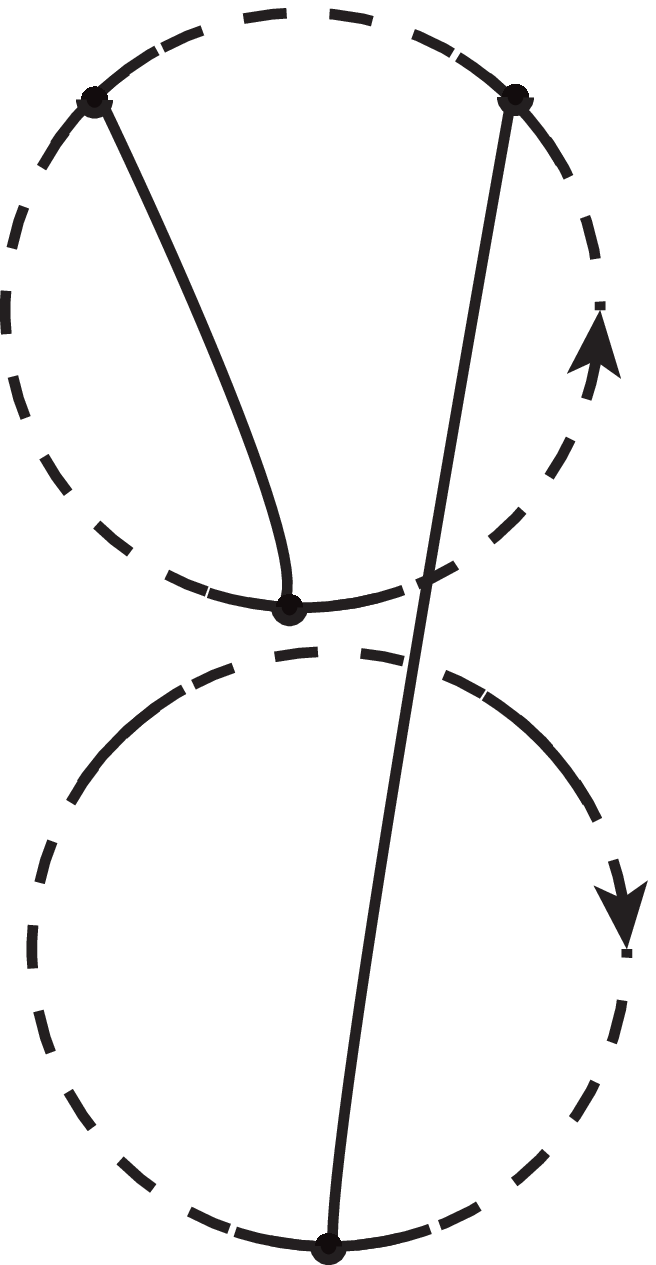}}}
\newcommand{\slt}{\raisebox{-0.44\height}{\includegraphics[width=1.2cm]{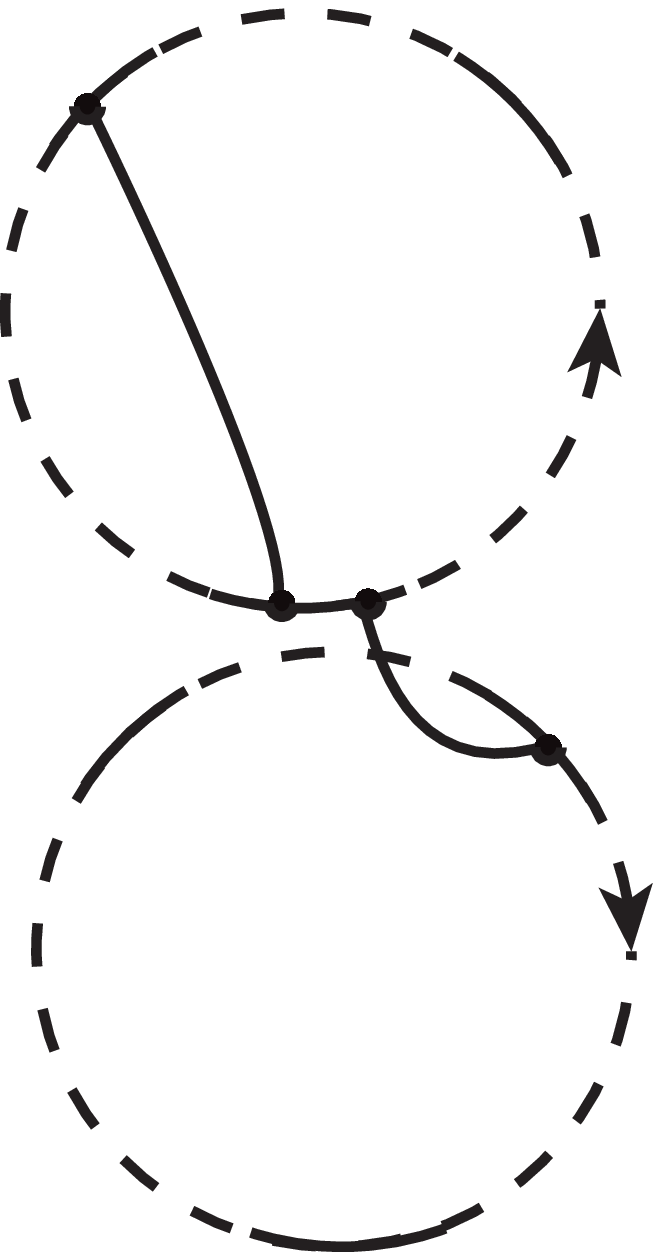}}}
\newcommand{\slth}{\raisebox{-0.44\height}{\includegraphics[width=1.2cm]{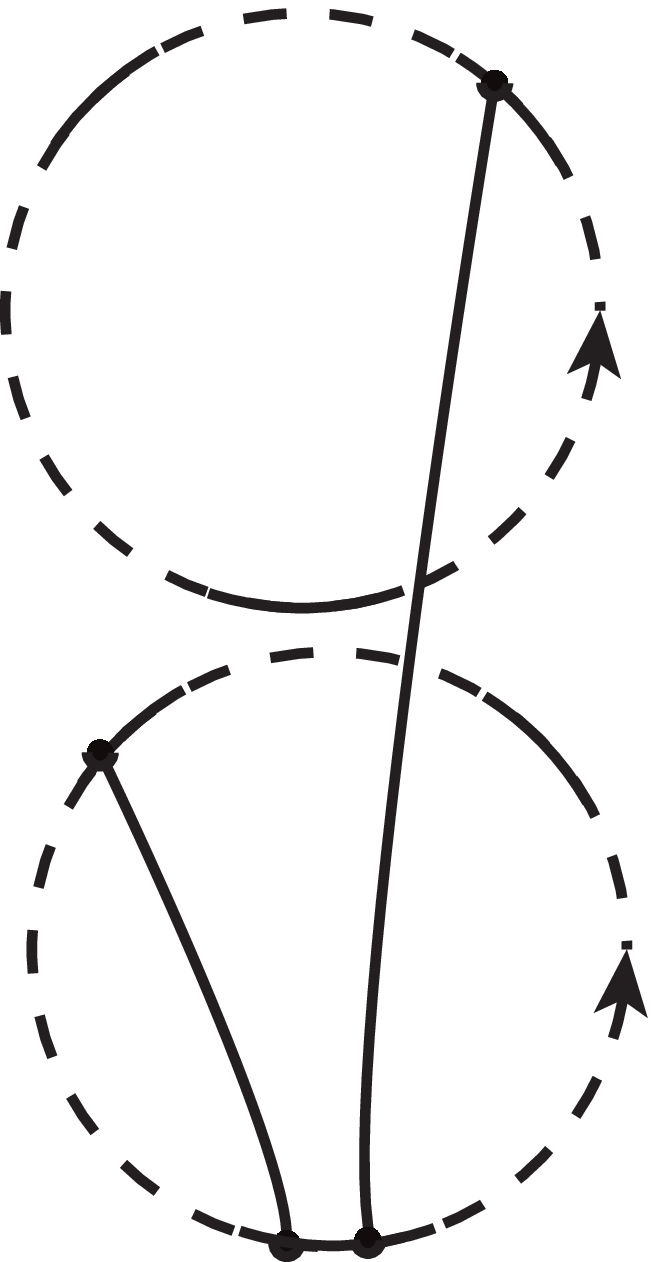}}}
\newcommand{\slf}{\raisebox{-0.44\height}{\includegraphics[width=1.2cm]{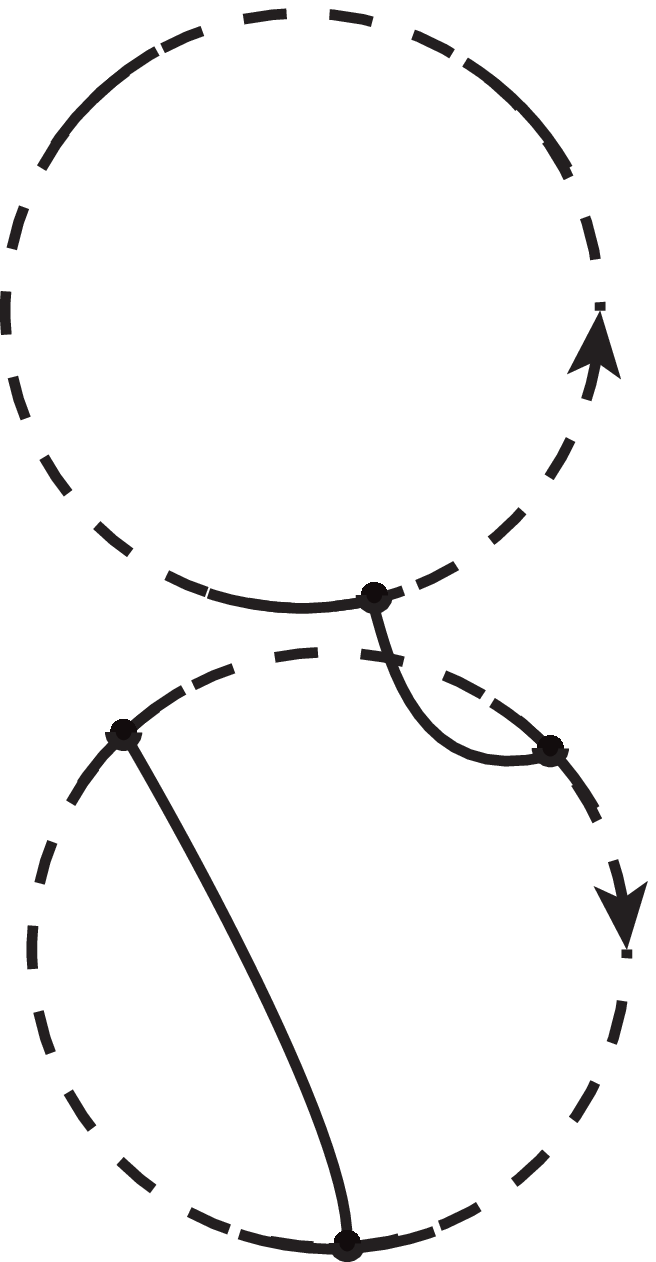}}}

\newcommand{\fr}{\raisebox{-0.35\height}{\includegraphics[width=1.0cm]{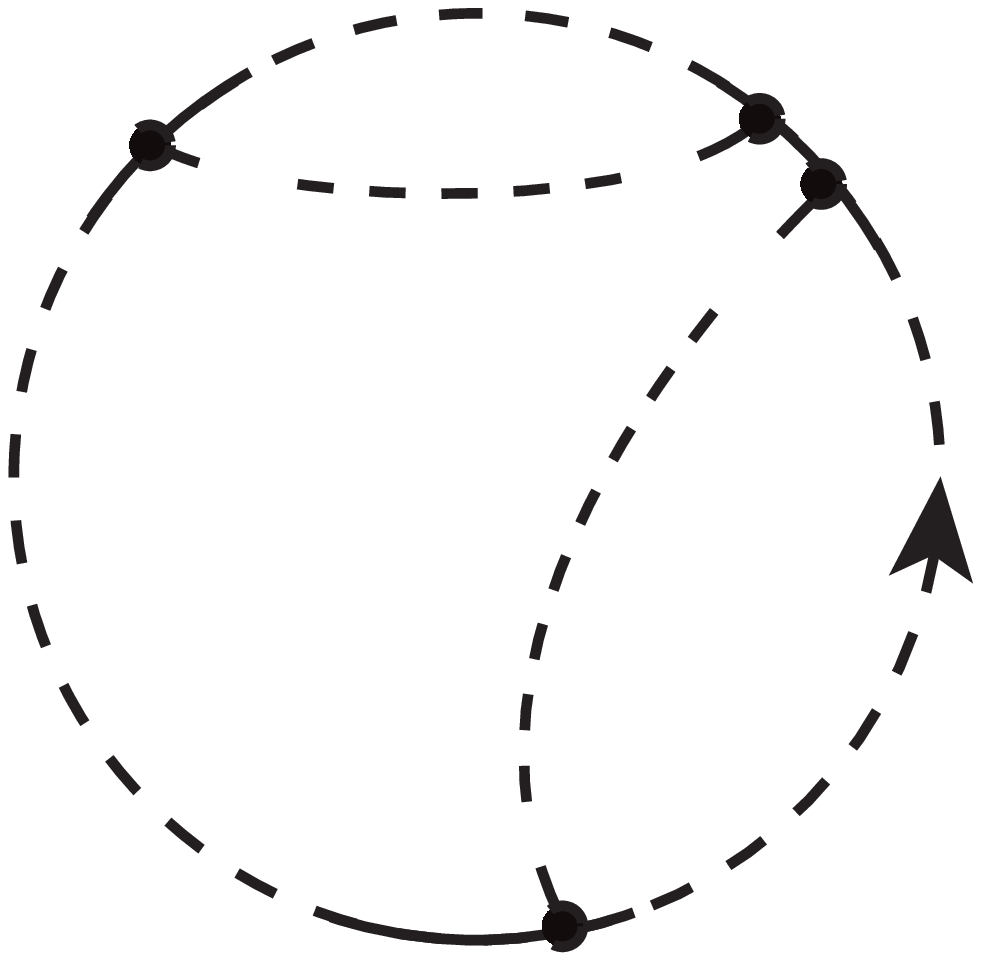}}}
\newcommand{\fro}{\raisebox{-0.44\height}{\includegraphics[width=1.2cm]{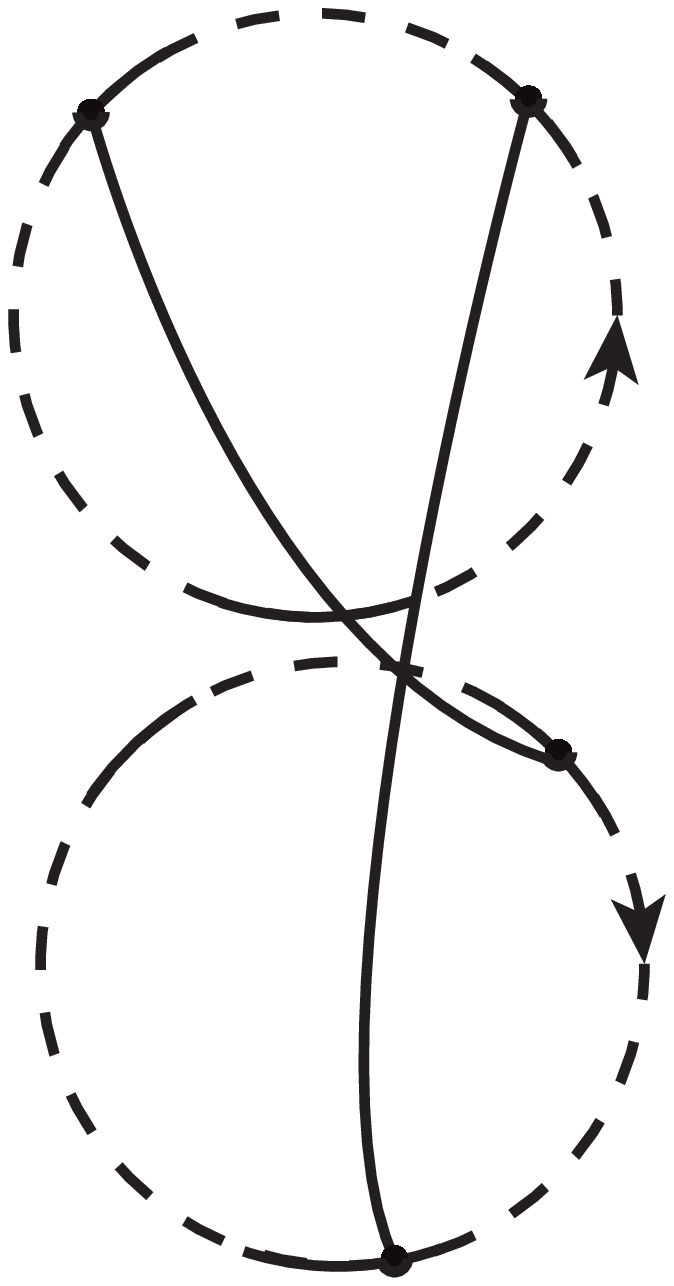}}}
\newcommand{\frt}{\raisebox{-0.44\height}{\includegraphics[width=1.2cm]{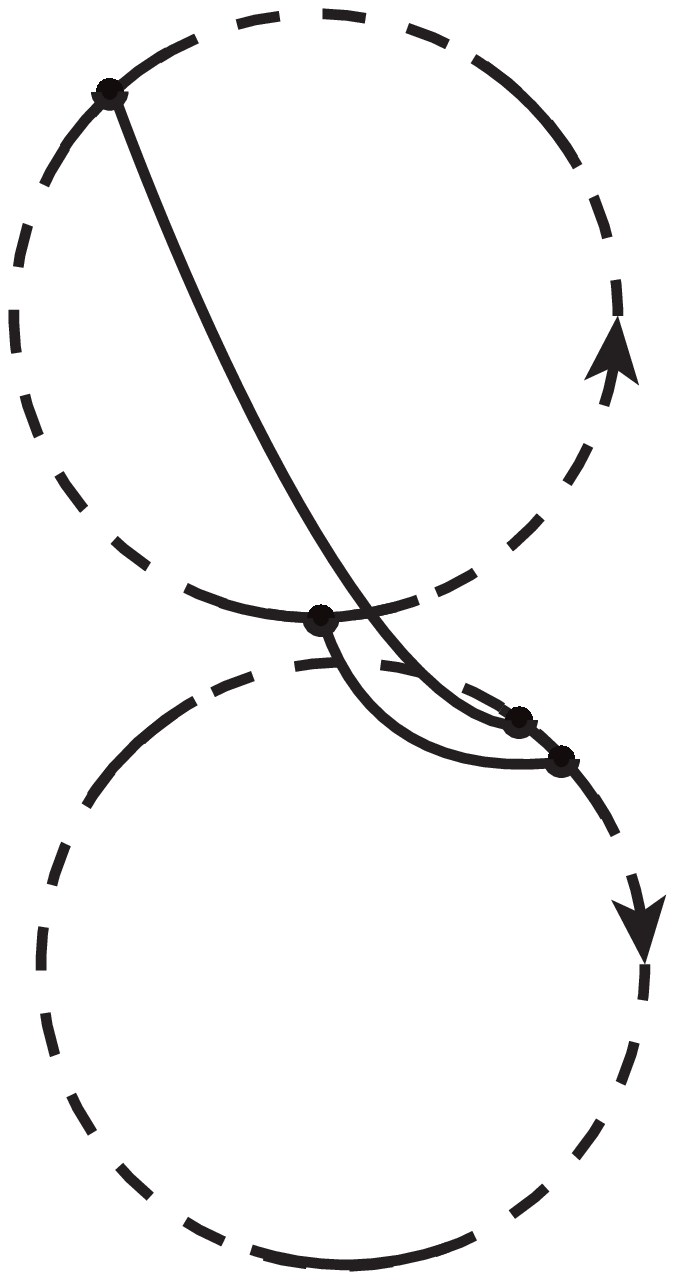}}}
\newcommand{\frth}{\raisebox{-0.44\height}{\includegraphics[width=1.2cm]{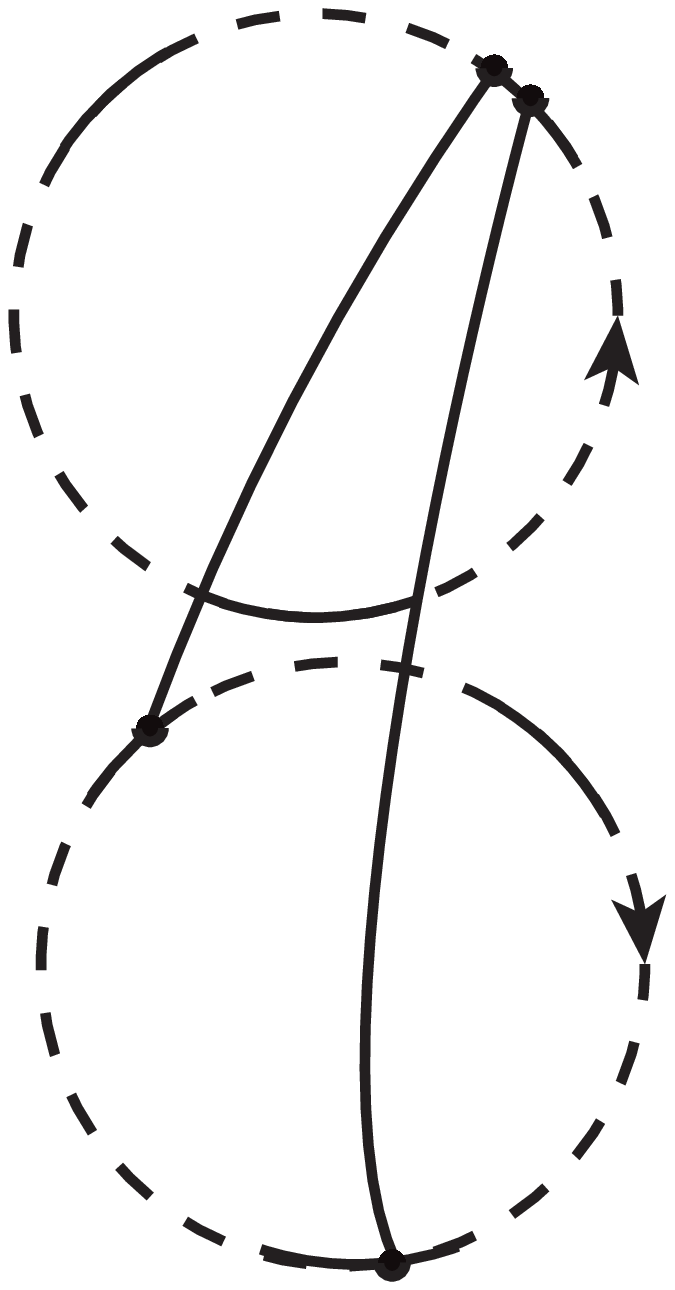}}}
\newcommand{\frf}{\raisebox{-0.44\height}{\includegraphics[width=1.2cm]{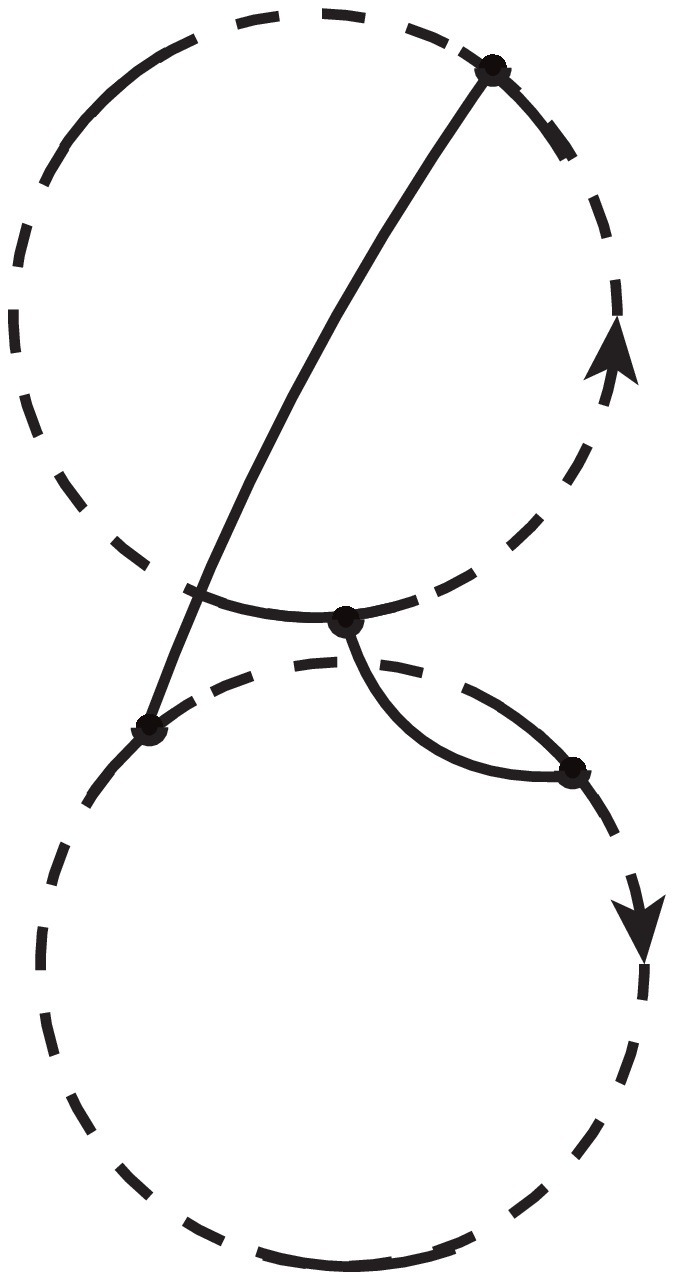}}}

\newcommand{\sref}{\raisebox{-0.35\height}{\includegraphics[width=1.0cm]{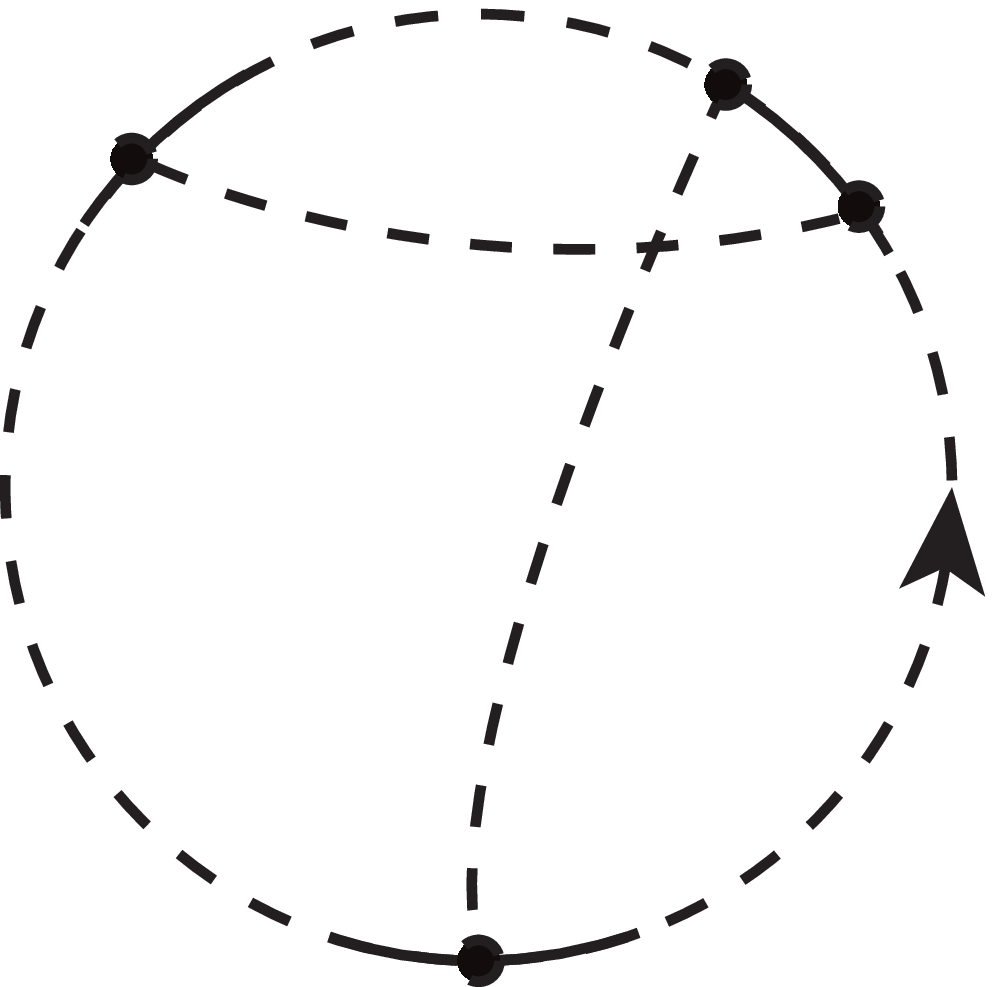}}}
\newcommand{\sro}{\raisebox{-0.44\height}{\includegraphics[width=1.2cm]{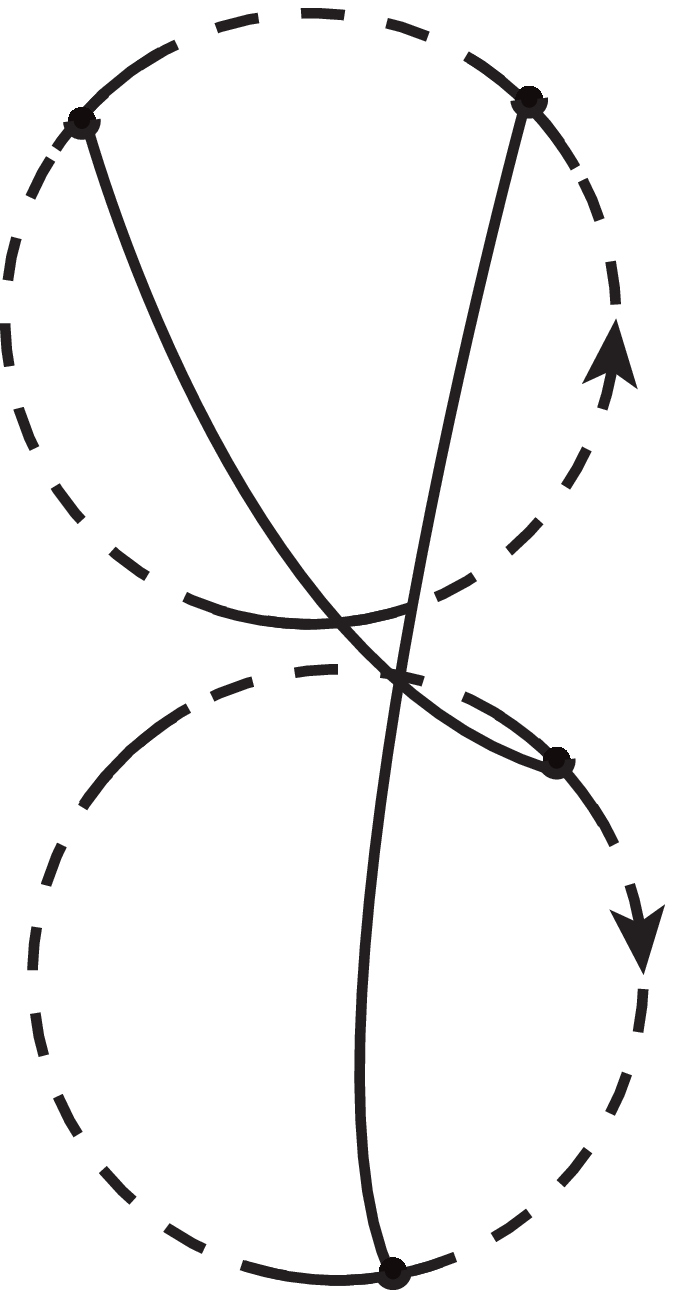}}}
\newcommand{\srt}{\raisebox{-0.44\height}{\includegraphics[width=1.2cm]{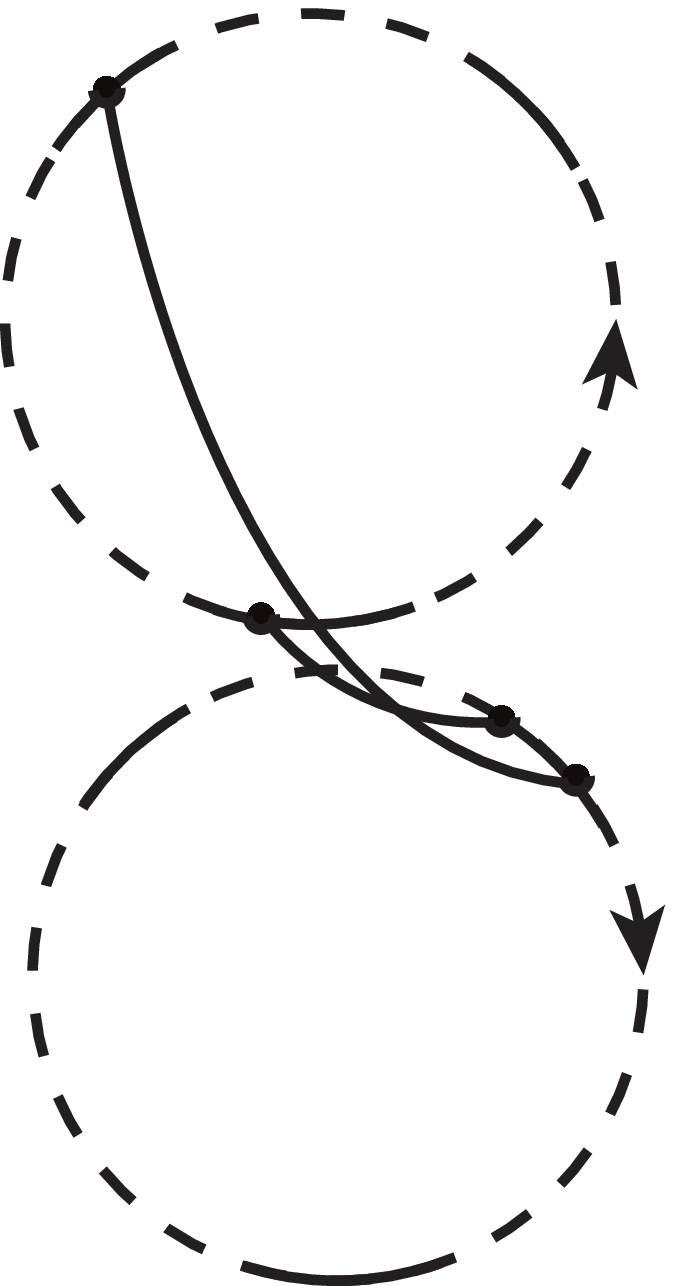}}}
\newcommand{\srth}{\raisebox{-0.44\height}{\includegraphics[width=1.2cm]{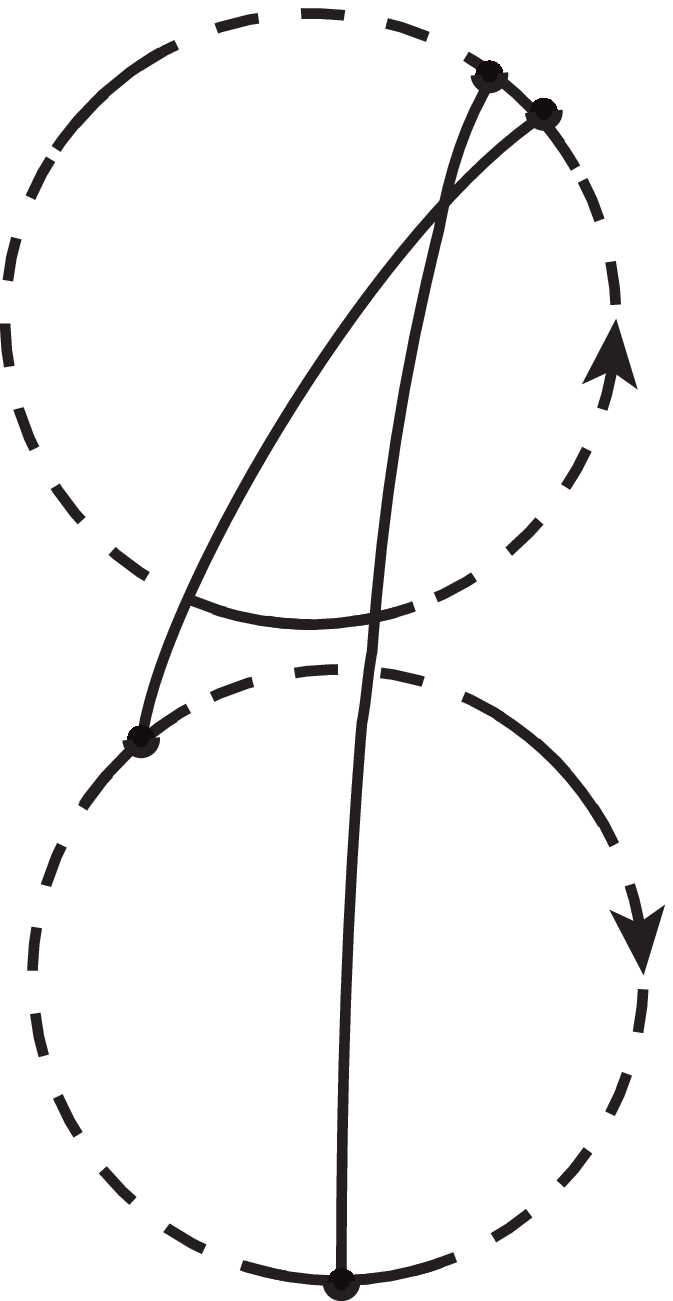}}}
\newcommand{\srf}{\raisebox{-0.44\height}{\includegraphics[width=1.2cm]{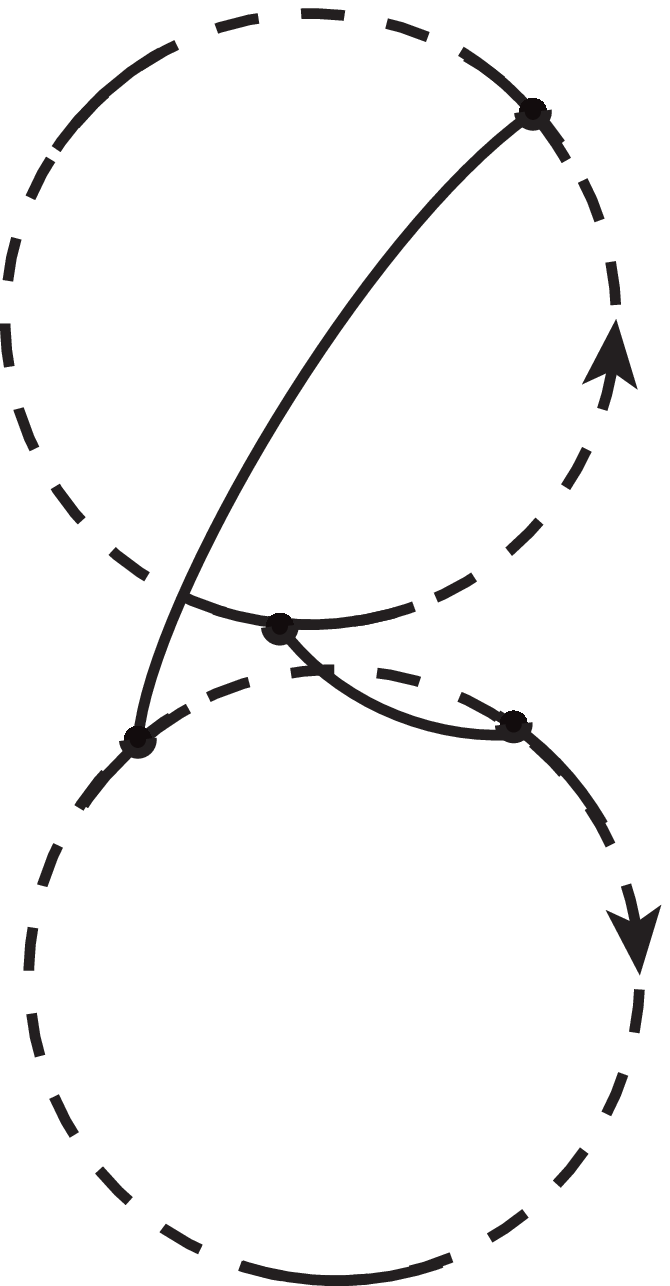}}}

\newcommand{\twocsone}{\raisebox{-0.44\height}{\includegraphics[width=1.2cm]{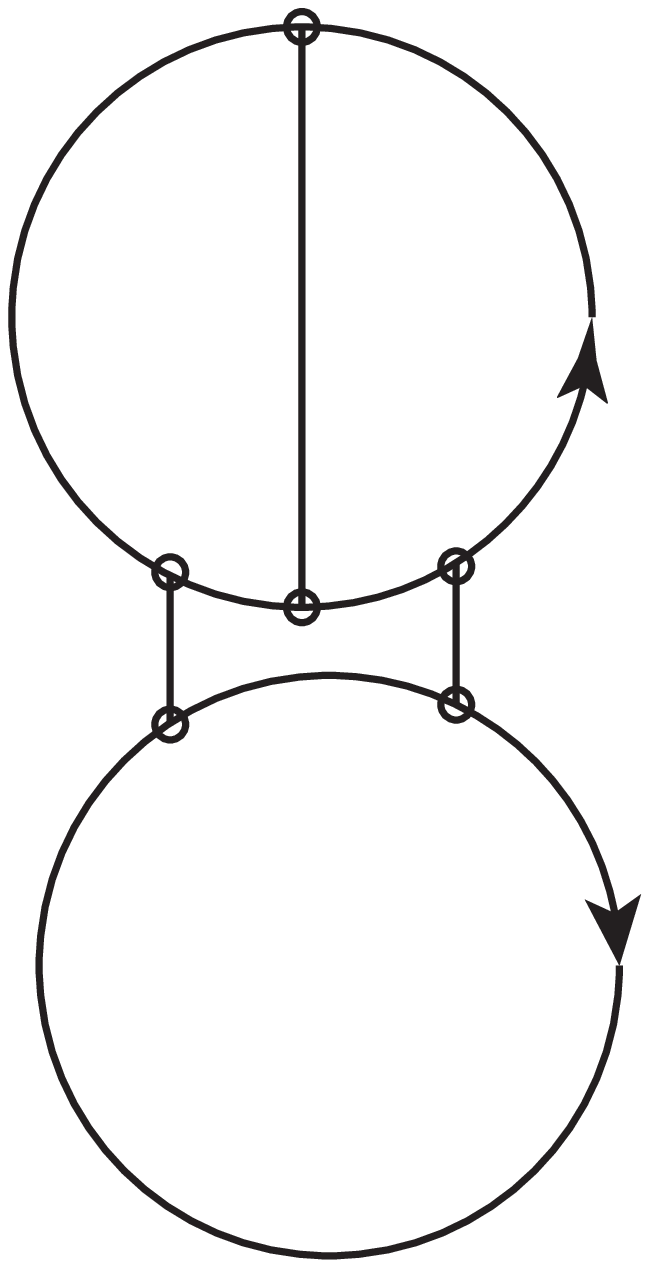}}}
\newcommand{\twocstwo}{\raisebox{-0.44\height}{\includegraphics[width=1.2cm]{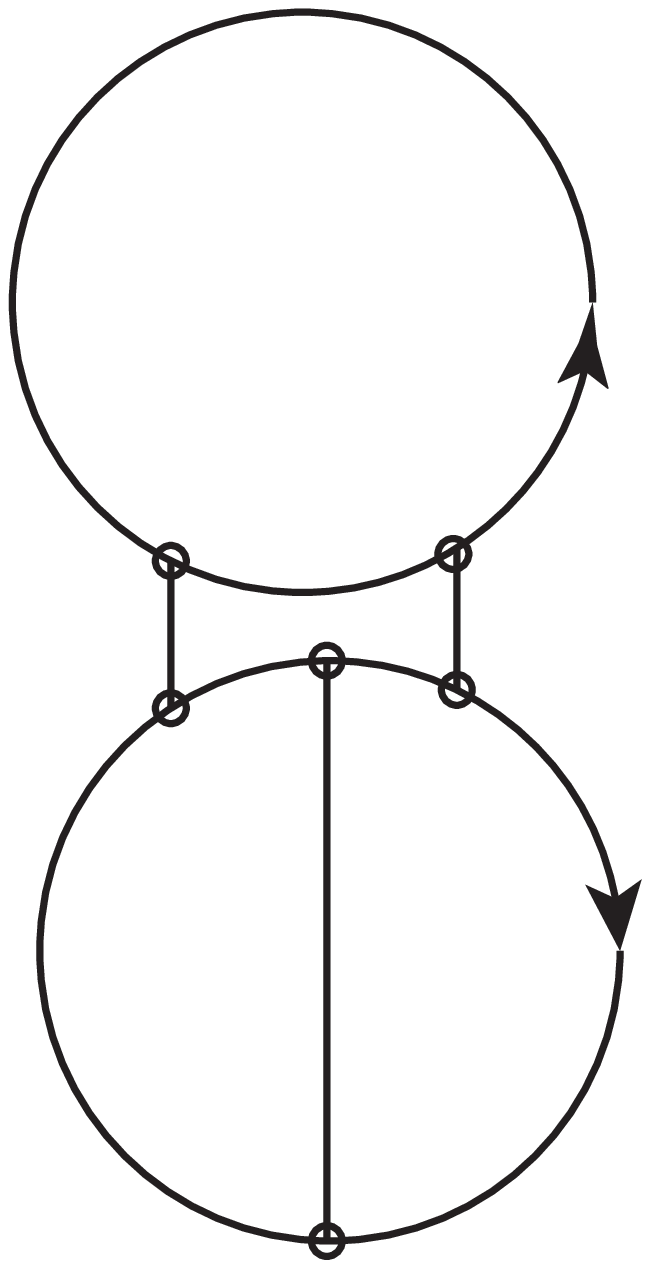}}}

\newcommand{\twocsthree}{\raisebox{-0.44\height}{\includegraphics[width=1.2cm]{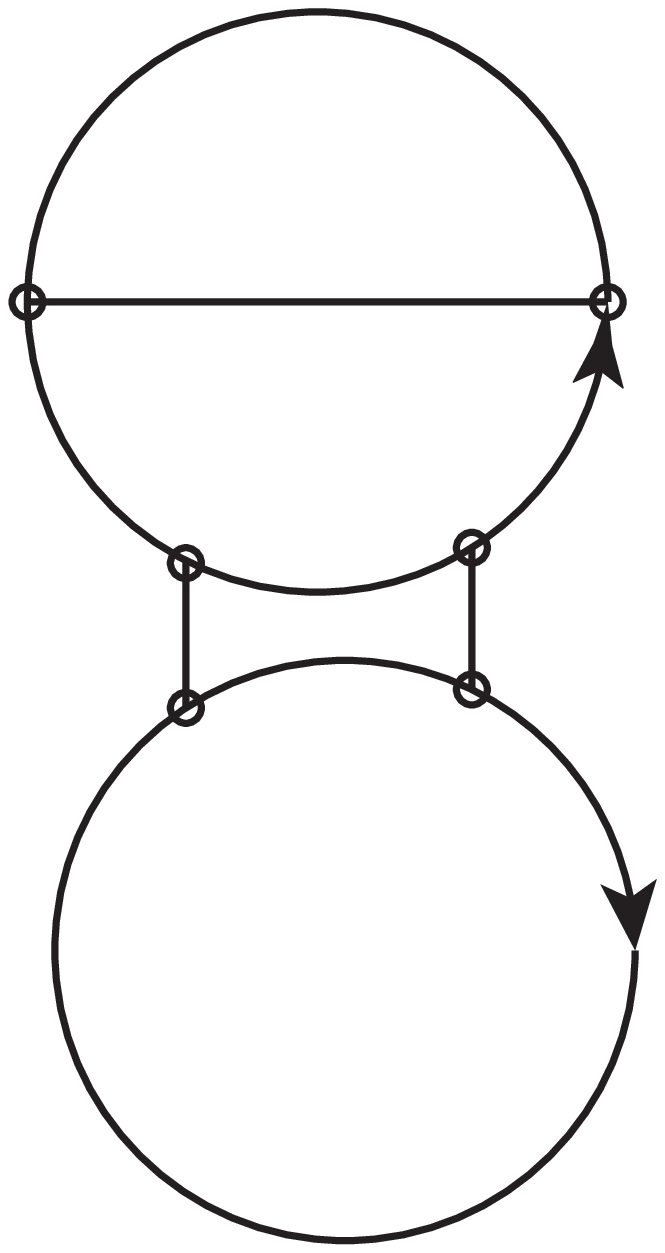}}}
\newcommand{\twocsfour}{\raisebox{-0.44\height}{\includegraphics[width=1.2cm]{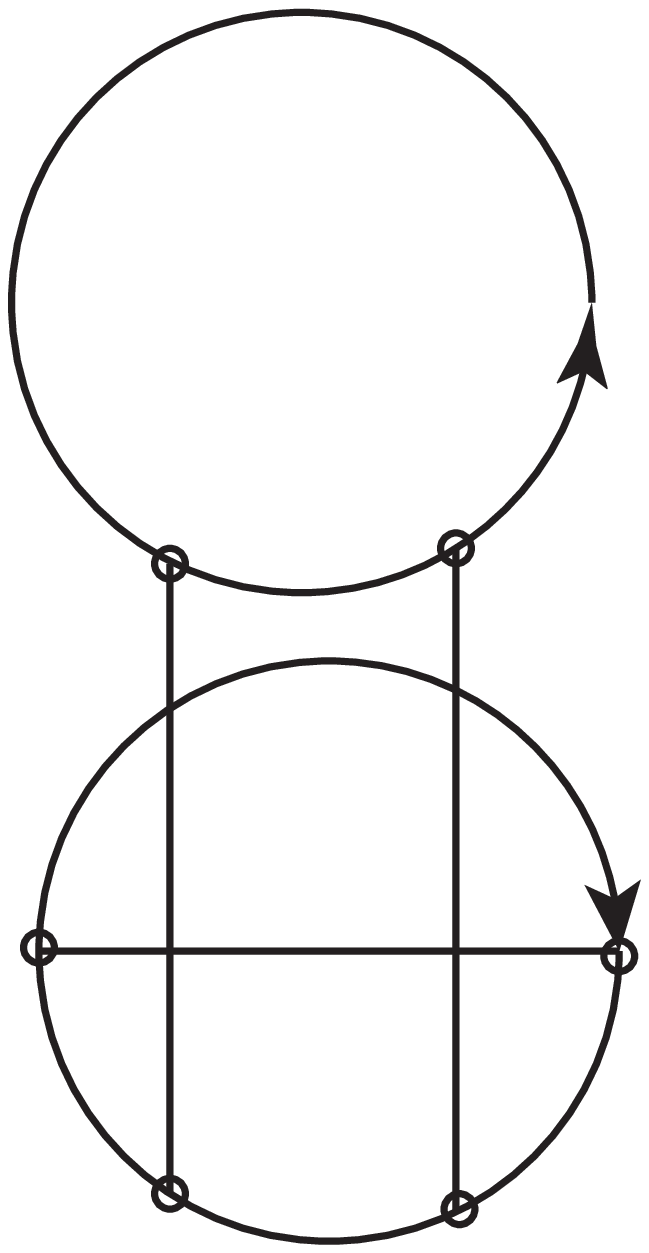}}}

\newcommand{\exfld}{\raisebox{-0.13\height}{\includegraphics[width=2.3cm]{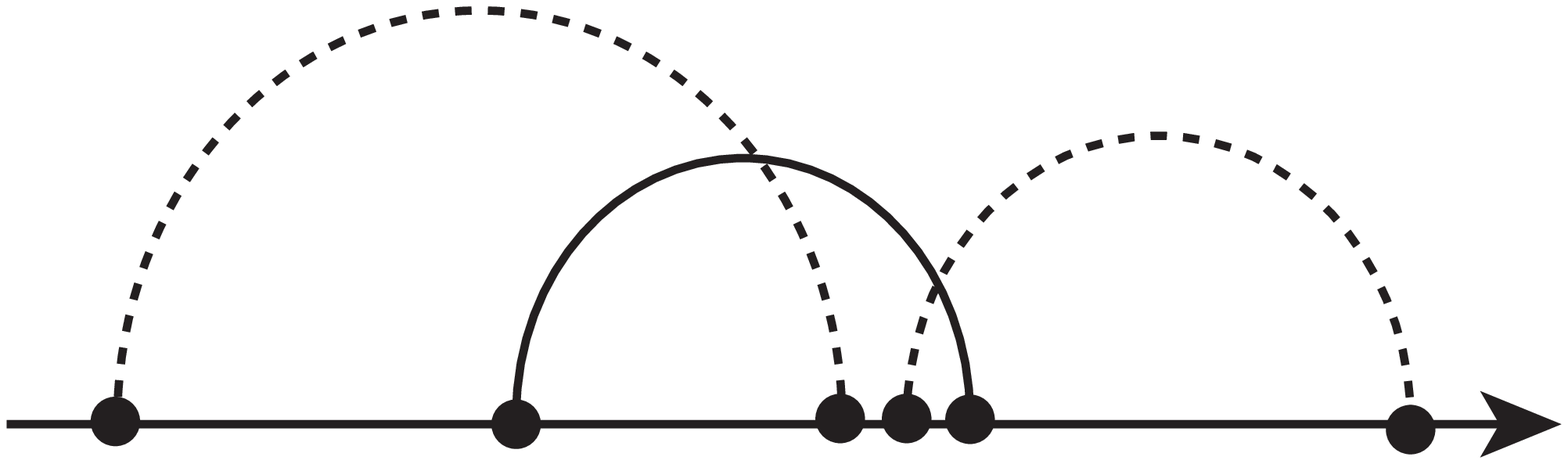}}}
\newcommand{\edldo}{\raisebox{-0.17\height}{\includegraphics[width=2.5cm]{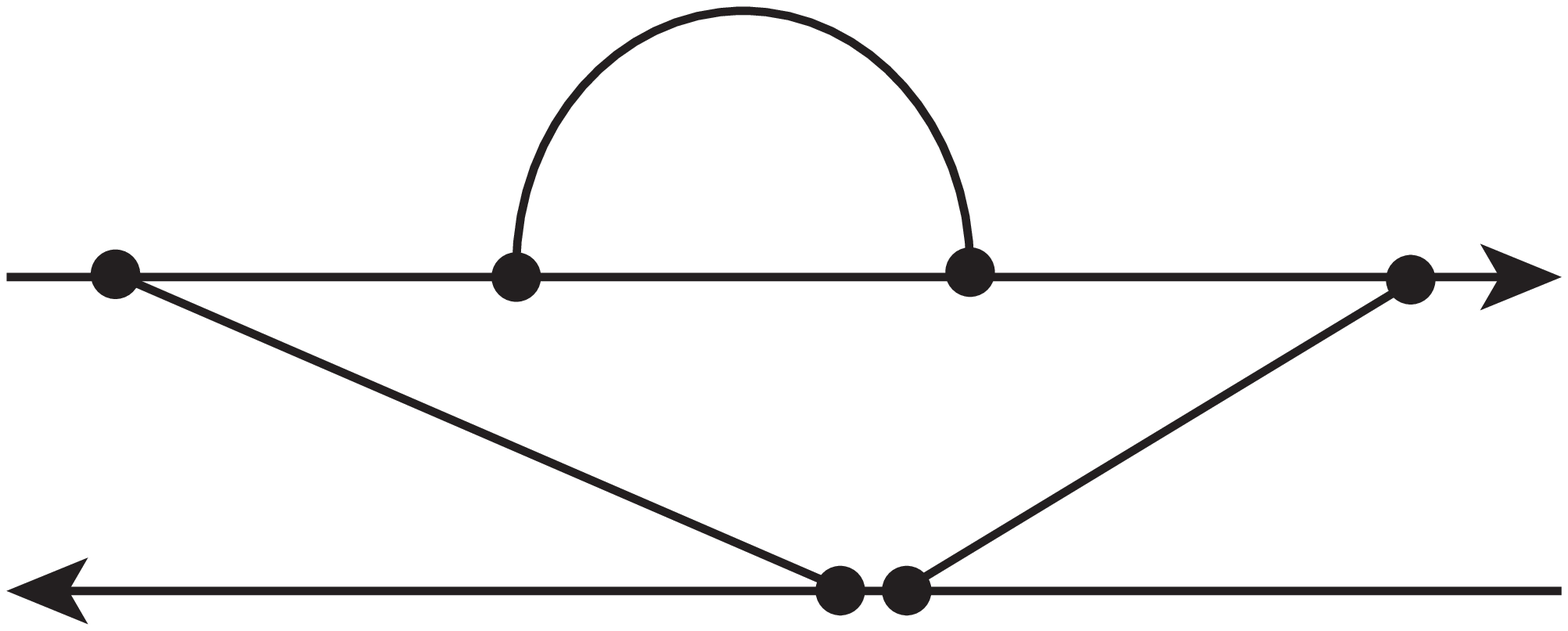}}}
\newcommand{\edldt}{\raisebox{-0.17\height}{\includegraphics[width=2.5cm]{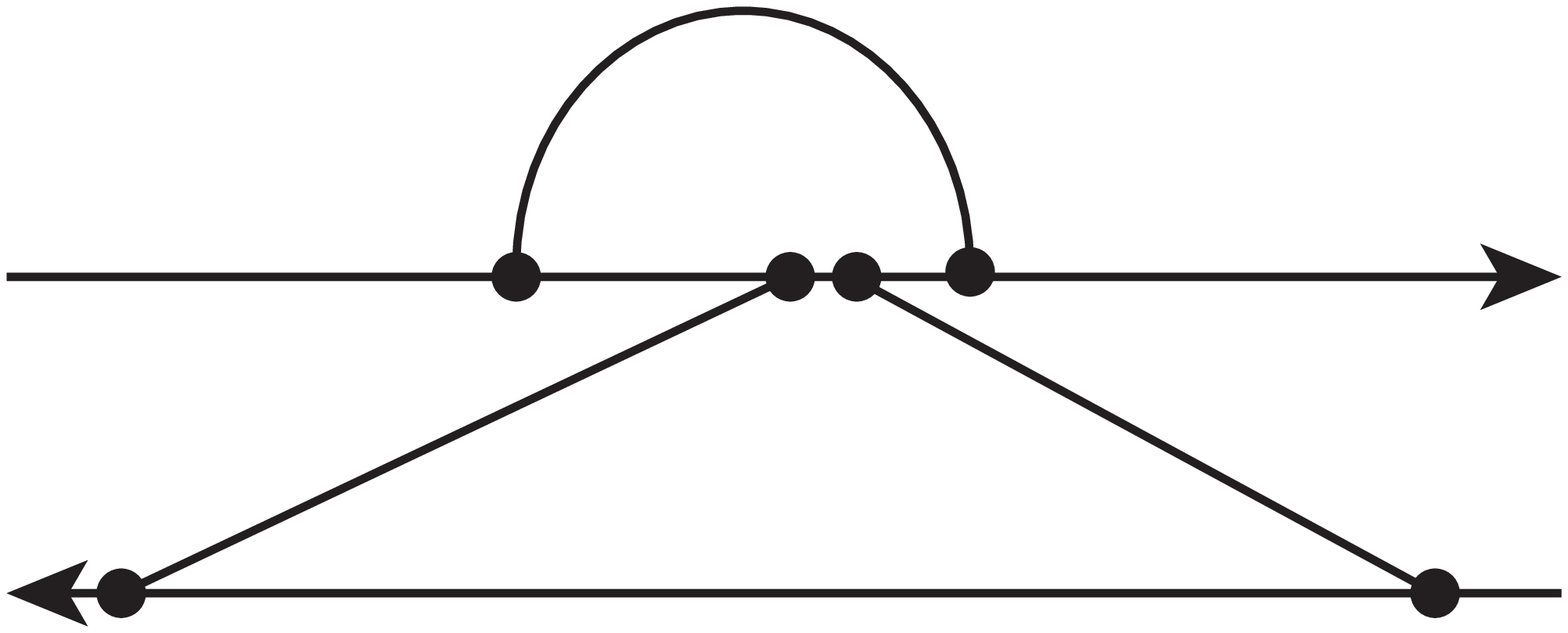}}}
\newcommand{\edldth}{\raisebox{-0.17\height}{\includegraphics[width=2.5cm]{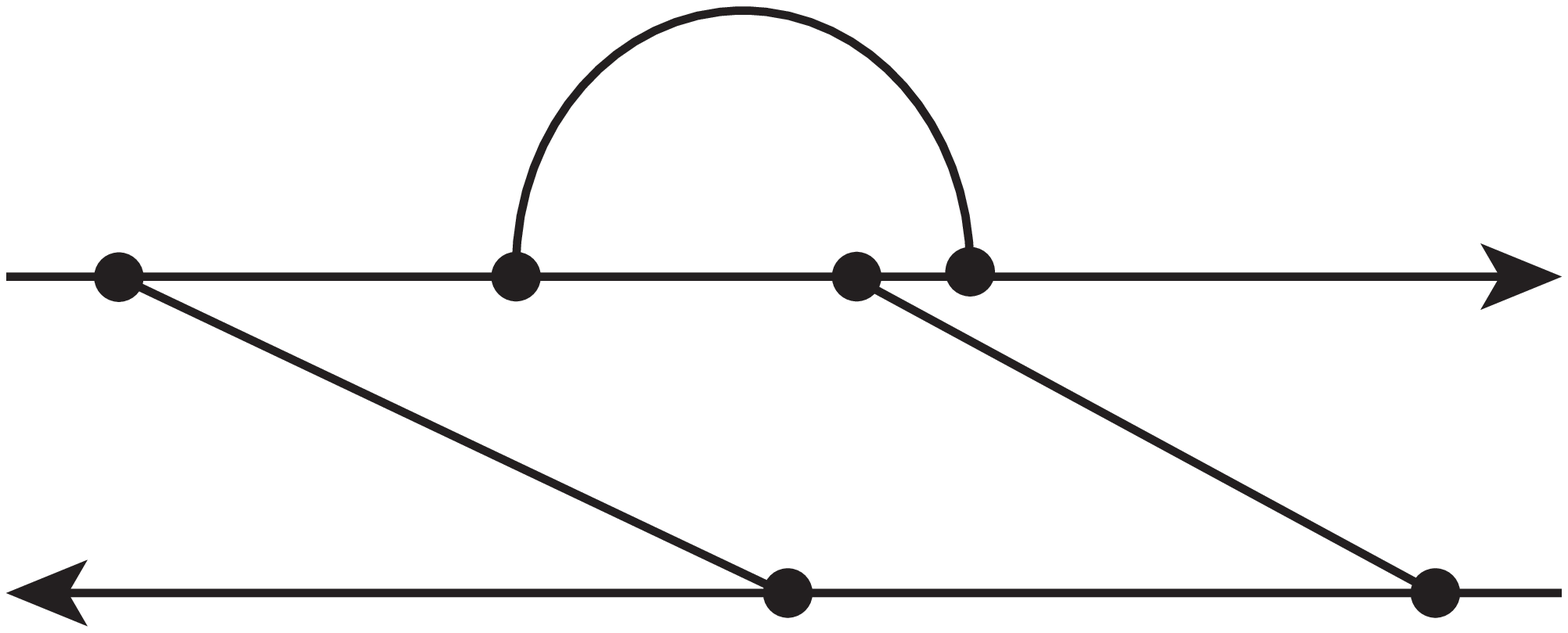}}}
\newcommand{\edldf}{\raisebox{-0.17\height}{\includegraphics[width=2.5cm]{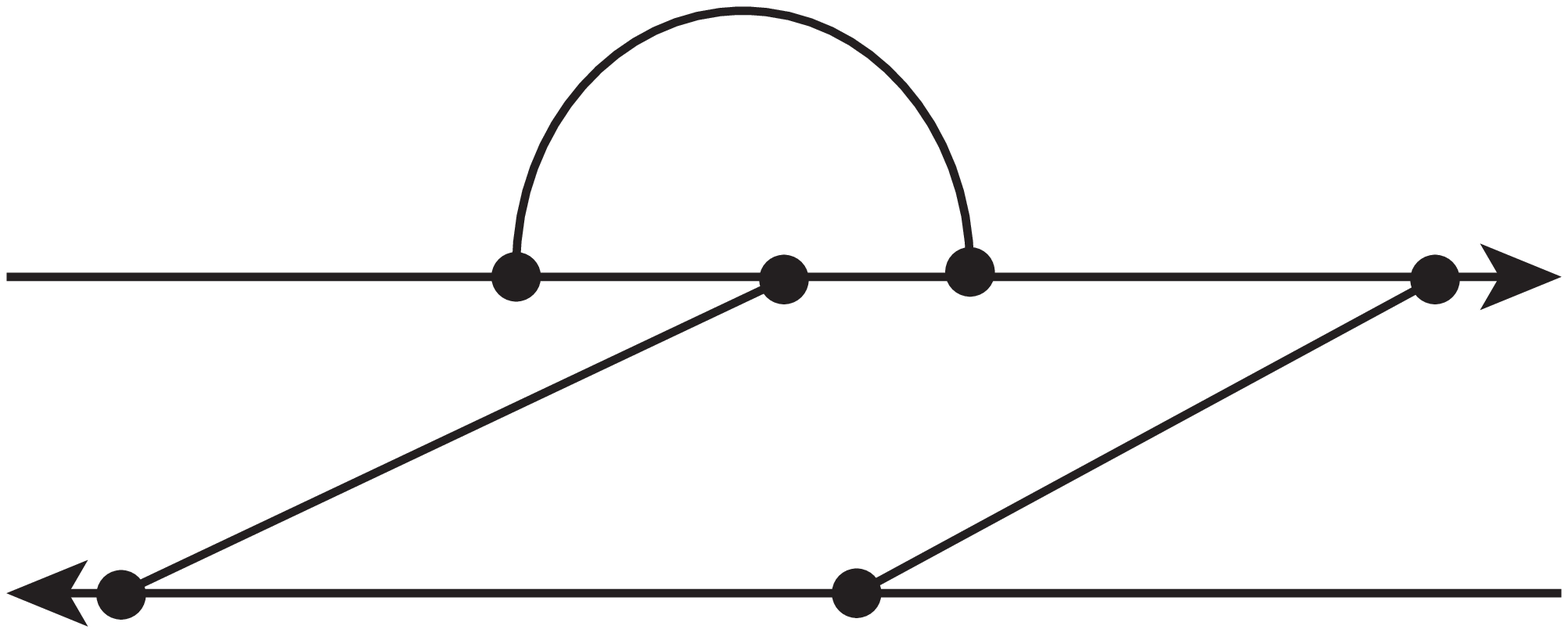}}}
\newcommand{\edldfv}{\raisebox{-0.26\height}{\includegraphics[width=2.5cm]{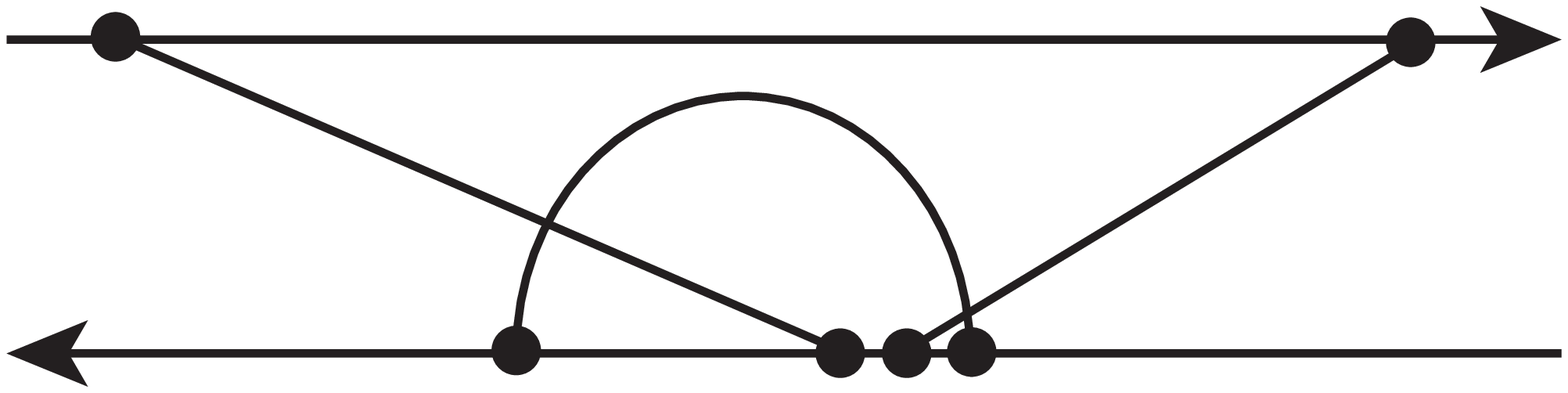}}}
\newcommand{\edlds}{\raisebox{-0.26\height}{\includegraphics[width=2.5cm]{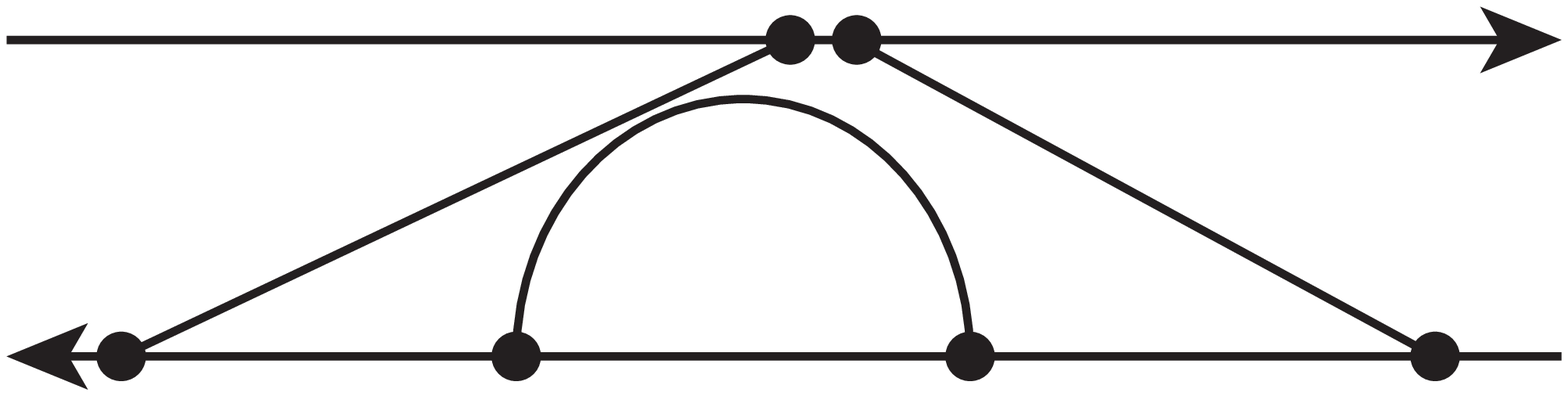}}}
\newcommand{\edldse}{\raisebox{-0.26\height}{\includegraphics[width=2.5cm]{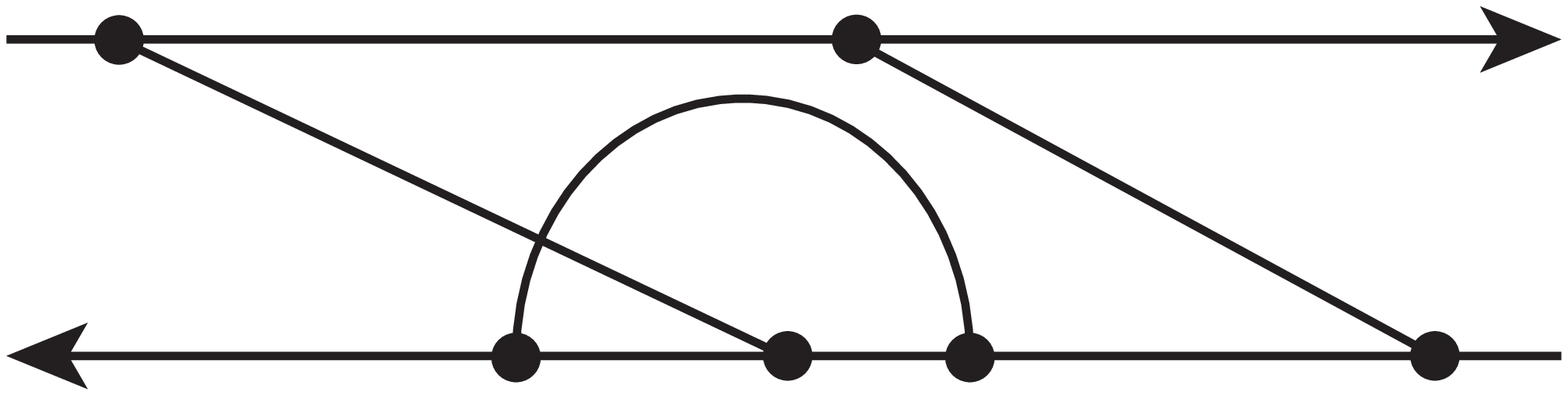}}}
\newcommand{\edlde}{\raisebox{-0.26\height}{\includegraphics[width=2.5cm]{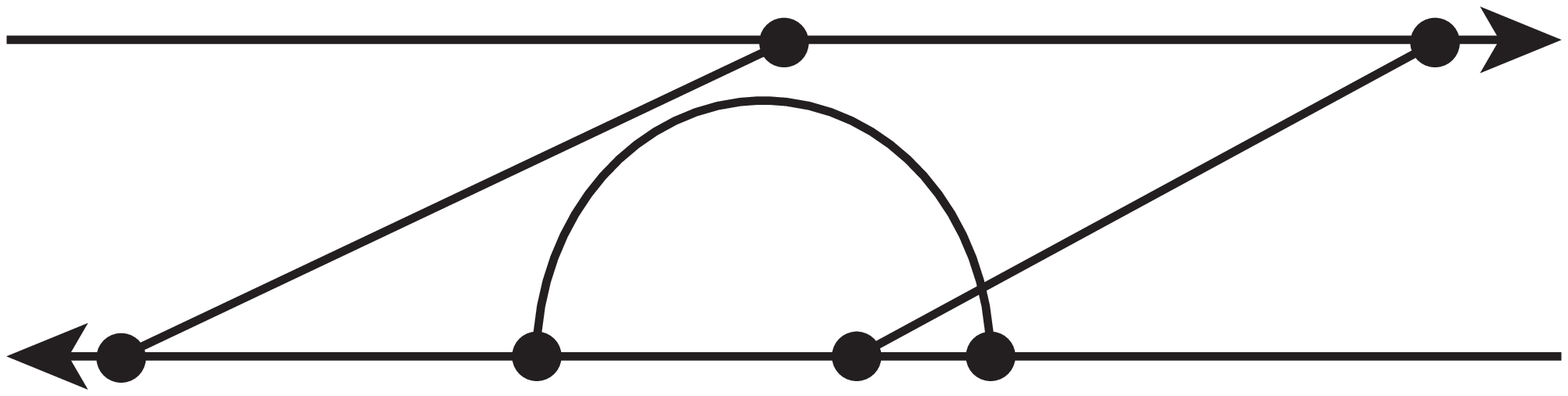}}}

\newcommand{\lfdone}{\raisebox{-0.17\height}{\includegraphics[width=2.5cm]{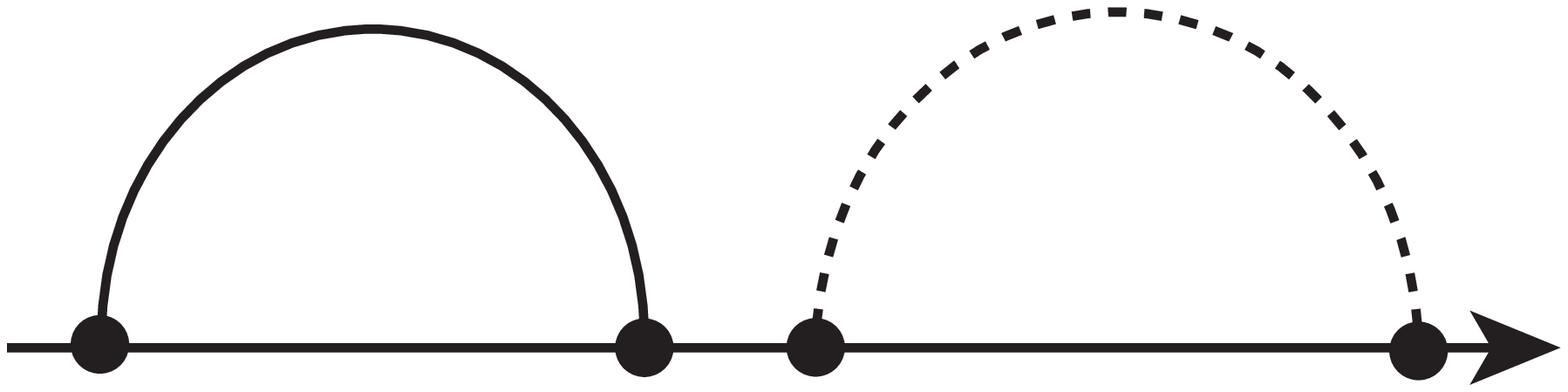}}}
\newcommand{\lfdtwo}{\raisebox{-0.17\height}{\includegraphics[width=2.5cm]{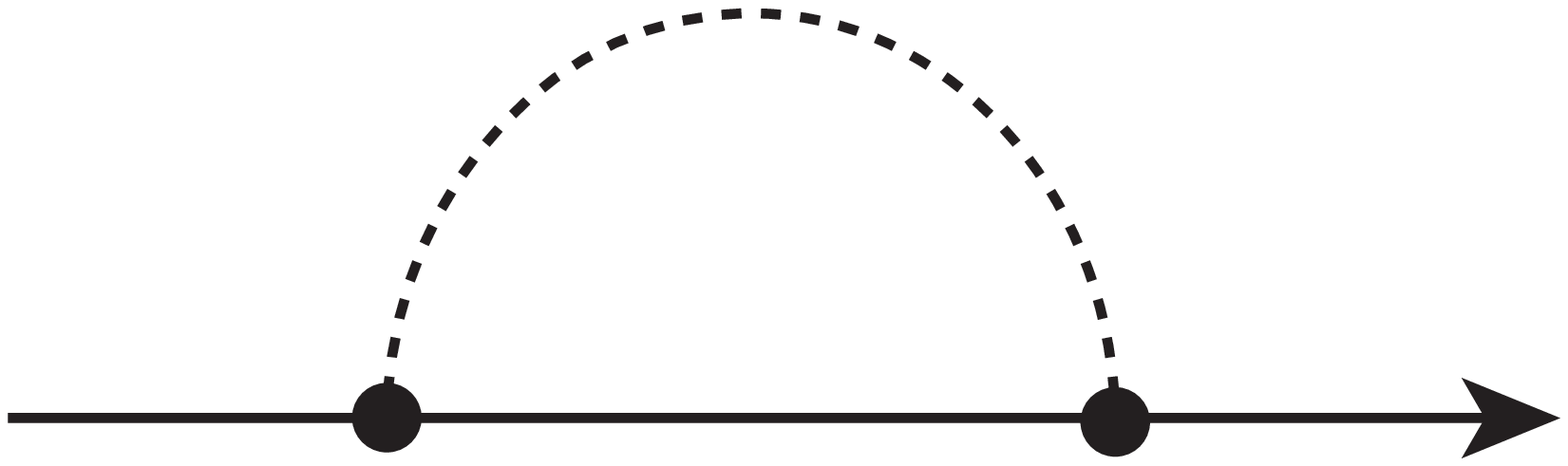}}}

\newcommand{\fconsum}{\raisebox{-0.17\height}{\includegraphics[width=3.5cm]{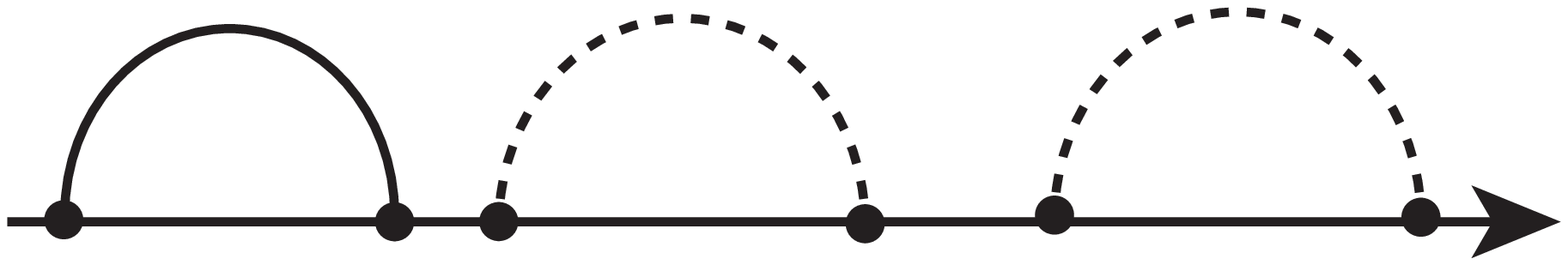}}}
\newcommand{\sconsum}{\raisebox{-0.17\height}{\includegraphics[width=3.5cm]{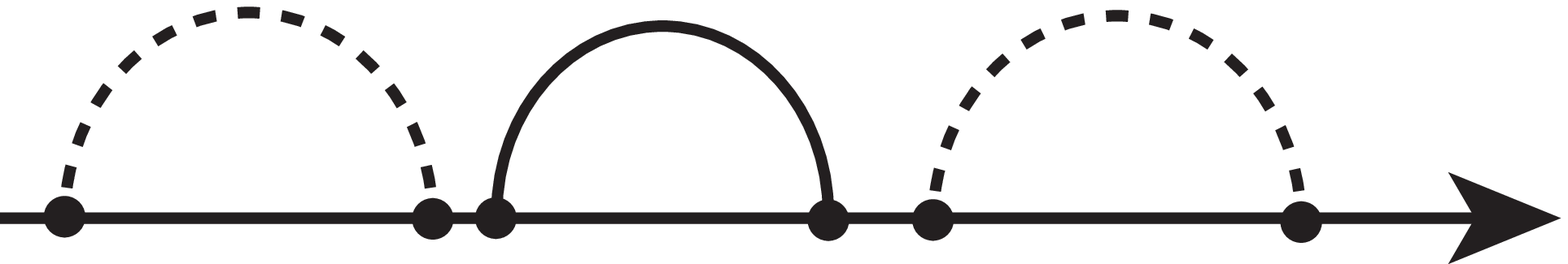}}}

\title{A parity map of framed chord diagrams}

\author{Denis Petrovich Ilyutko\footnote{Partially supported by grants of RF
President NSh -- 1410.2012.1, RFBR 13-01-00664-a, 13-01-00830-a,
14-01-91161 and 14-01-31288.}, Vassily Olegovich
Manturov\footnote{Partially supported by grants of  the Russian
Government 14.Z50.31.0020, RF President NSh -- 1410.2012.1, RFBR
13-01-00830-a 14-01-91161 and 14-01-31288.}}

\begin{document}
\date{}

\maketitle

\abstract{We consider framed chord diagrams, i.e.\ chord diagrams
with chords of two types. It is well known that chord diagrams
modulo $4$T-relations admit Hopf algebra structure, where the
multiplication is given by any connected sum with respect to the
orientation. But in the case of framed chord diagrams a natural way
to define a multiplication is not known yet. In the present paper,
we first define a new module $\mathcal{M}_2$ which is generated by
chord diagrams on two circles and factored by $4$T-relations. Then
we construct a ``parity'' map from the module of framed chord
diagrams into $\mathcal{M}_2$ and a weight system on
$\mathcal{M}_2$. Using the map and weight system we show that a
connected sum for framed chord diagrams is not a well-defined
operation. In the end of the paper we touch linear diagrams, the
circle replaced by a directed line.}

%%%%%%%%%%%%%%%%%%%%%%%%%%%%%%%%%%%%%%%%%%%%%%%%%%%%%%%%%%%%
%%%%%%%%%%%%%%%%%%%%%%%%%%%%%%%%%%%%%%%%%%%%%%%%%%%%%%%%%%%%

 \section {Introduction}

%%%%%%%%%%%%%%%%%%%%%%%%%%%%%%%%%%%%%%%%%%%%%%%%%%%%%%%%%%%%
%%%%%%%%%%%%%%%%%%%%%%%%%%%%%%%%%%%%%%%%%%%%%%%%%%%%%%%%%%%%

In  topology and graph theory, many notions often have their
``odd'', ``non-orientable'', ``framed'' counterparts, see for
example~\cite{Gor,IMN1,IMN2,Lan,Mant31,Mant32,Mant33,Mant34,Mant35,Mant37,Mant36,Mant38}.
Usually, even objects are better understood, however, in the odd
case, it is much easier to catch the non-trivial information. The
interrelation between the ``even'' and ``odd'' parts usually relies
upon some functorial mappings, coverings etc., and allows one to
understand better the real reason of various
effects~\cite{Mant35,Mant40,Mant41}.

It suffices to mention the three faces of the chord diagram theory:
the Gauss diagram approach and the rotating circuit approach for
four-valent graph~\cite{I-ko:f4vg,I-ko:eq,Mant22,Mant30,Nik_pVc},
and $J$-invariants of plane curves~\cite{Lan}. Also framed chord
diagram are used for describing the combinatorics of a non-generic
Legendrian knots in some $3$-manifolds.

The most famous face of chord diagrams is the role they play in the
chord diagram algebra~\cite{BN,BNG,ChDL,ChDM,Sob}. The weight
systems (i.e., linear functions on this algebra), due to
Vassiliev--Kontsevich theorem, lead to Vassiliev invariants of
knots, see~\cite{BN,ChDL,ChDM}. Note that the multiplication in the
chord diagram algebra is defined by taking a connected sum of two
chord diagrams. This operation is well defined up to $4$T-relation.
In the case of framed chord diagrams one can define $4$T-relations
and consider a connected sum of two framed diagrams. An attempt to
prove that this operation is well defined up to the $4$T-relations
fails.

The main goal of the present paper is to construct a ``parity'' map
from the set of framed chord diagrams, factored by $4$T-relations,
to a set of objects where chords of only one type are used. Using
this map and some invariant we demonstrate two examples of framed
chord diagrams having different images under the map. Therefore, a
connected sum is not a well-defined operation in the set of framed
chord diagram up to $4$T-relations. Note that one ``forgetful'' map
was defined in~\cite{Kar} for constructing framed weight systems.

The structure of the paper is as follows. In the next section we
recall all necessary facts about framed chord diagrams. In
Sec.~\ref{sec:dchd} we introduce the notion of a double chord
diagram, i.e.\ a chord diagram on two oriented circles, and define a
weight system for the module generated by double chord diagram.
Section~\ref{sec:map} is devoted to a map from the framed chord
diagrams module to the double chord diagrams module. Then, in
Sec.~\ref{sec:appl} we apply this map and some weight system for
proving that a connected sum of two chord diagrams is not a
well-defined operations on the set of framed chord diagrams modulo
$4$T-relations. At the end of the paper we generalize all
constructions for the case of linear diagrams, i.e.\ chord diagrams
on a directed line instead of an oriented circle.

%%%%%%%%%%%%%%%%%%%%%%%%%%%%%%%%%%%%%%%%%%%%%%%%%%%%%%%%%%%%%%%%%%%%%%
\section{Framed chord diagrams}\label{sec:fchd}
%%%%%%%%%%%%%%%%%%%%%%%%%%%%%%%%%%%%%%%%%%%%%%%%%%%%%%%%%%%%%%%%%%%%%%

Throughout the paper, all graphs are finite. Let $G$ be a graph with
the set of vertices $V(G)$ and the set of edges $E(G)$. We say that
a vertex $v\in V(G)$ has {\em degree} $k$ if $v$ is incident to $k$
edges. A graph whose vertices have the same degree $k$ is called
{\em regular $k$-valent} or a {\em $k$-graph}. For any $k$, the free
loop, i.e.\ the graph without vertices, is considered as a
$k$-graph.

  \begin{definition}
A {\em chord diagram} is a cubic graph consisting of a selected
oriented Hamiltonian cycle (the {\em core circle}) and several
non-oriented edges ({\em chords}) connecting points on the core
circle in such a way that every point on the core circle is incident
to at most one chord. A chord diagram is {\em framed} if a map (a
{\em framing}) from the set of chords to $\mathbb{Z}/2\mathbb{Z}$ is
given, i.e.\ every chord is endowed with $0$ or $1$.
%Two chords of a chord diagram are called {\em linked} if the ends of one chord lie in different connected components of the core circle with the end-points of the second chord removed.
  \end{definition}

 \begin{remark}
We consider all framed chord diagrams up to orientation and framing
preserving isomorphisms of graphs taking one core circle to the
other one. In pictures the core circles of chord diagrams are
oriented in counterclockwise manner. Chords having  framing $0$ are
solid chords, and those having framing $1$, are dashed ones, see
Fig.~\ref{exch}.
 \end{remark}

 \begin{figure}
  \centering\includegraphics[width=100pt]{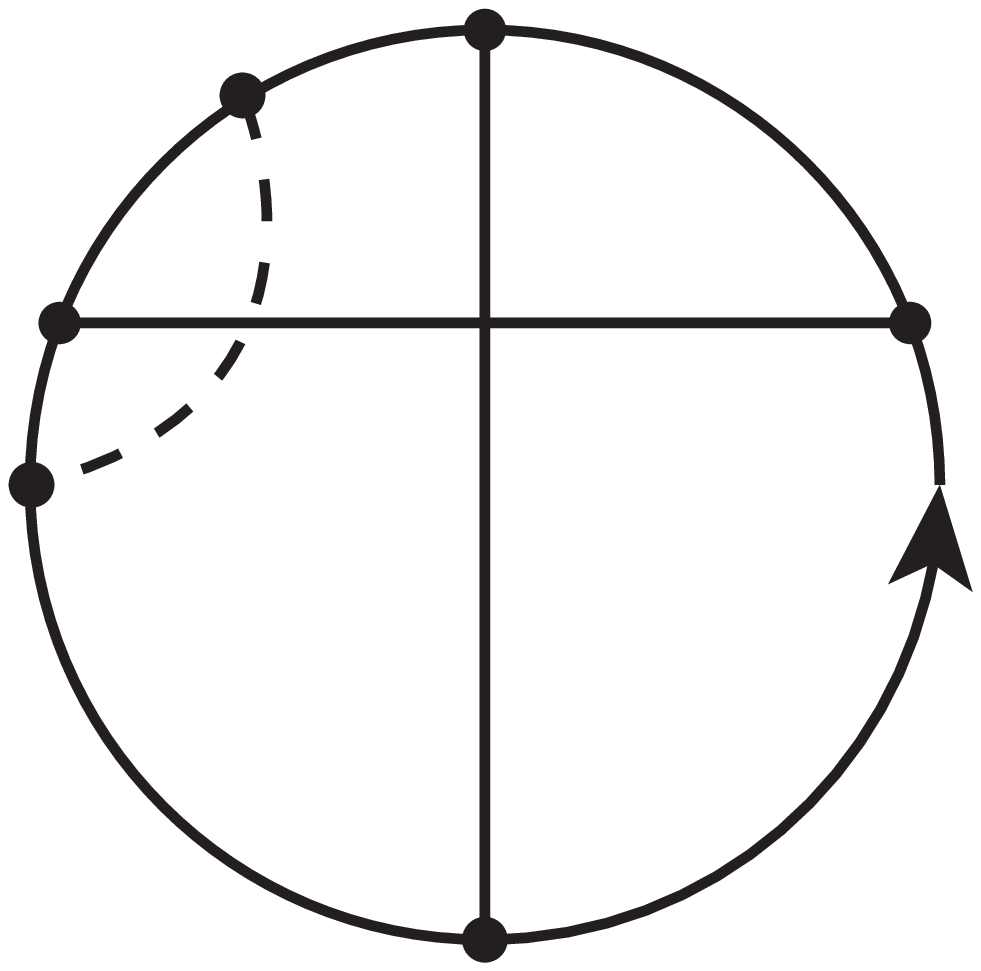}
  \caption{A framed chord diagram with two chords of framing $0$ and a chord of framing $1$}\label{exch}
 \end{figure}

Let $M^f$ be the free $\mathbb{Z}$-module generated by all framed
chord diagrams. Each element of $M^f$ is a finite linear combination
of framed chord diagrams with integer coefficients.

 \begin{definition}
The module $\mathcal{M}^f$ \cite{Lan} of framed chord diagrams is
the quotient module of $M^f$ modulo the relations shown in
Fig.~\ref{4trel}. We refer to these relations as to {\em
$4$T-relations}.
 \end{definition}

 \begin{figure}
  \centering\includegraphics[width=300pt]{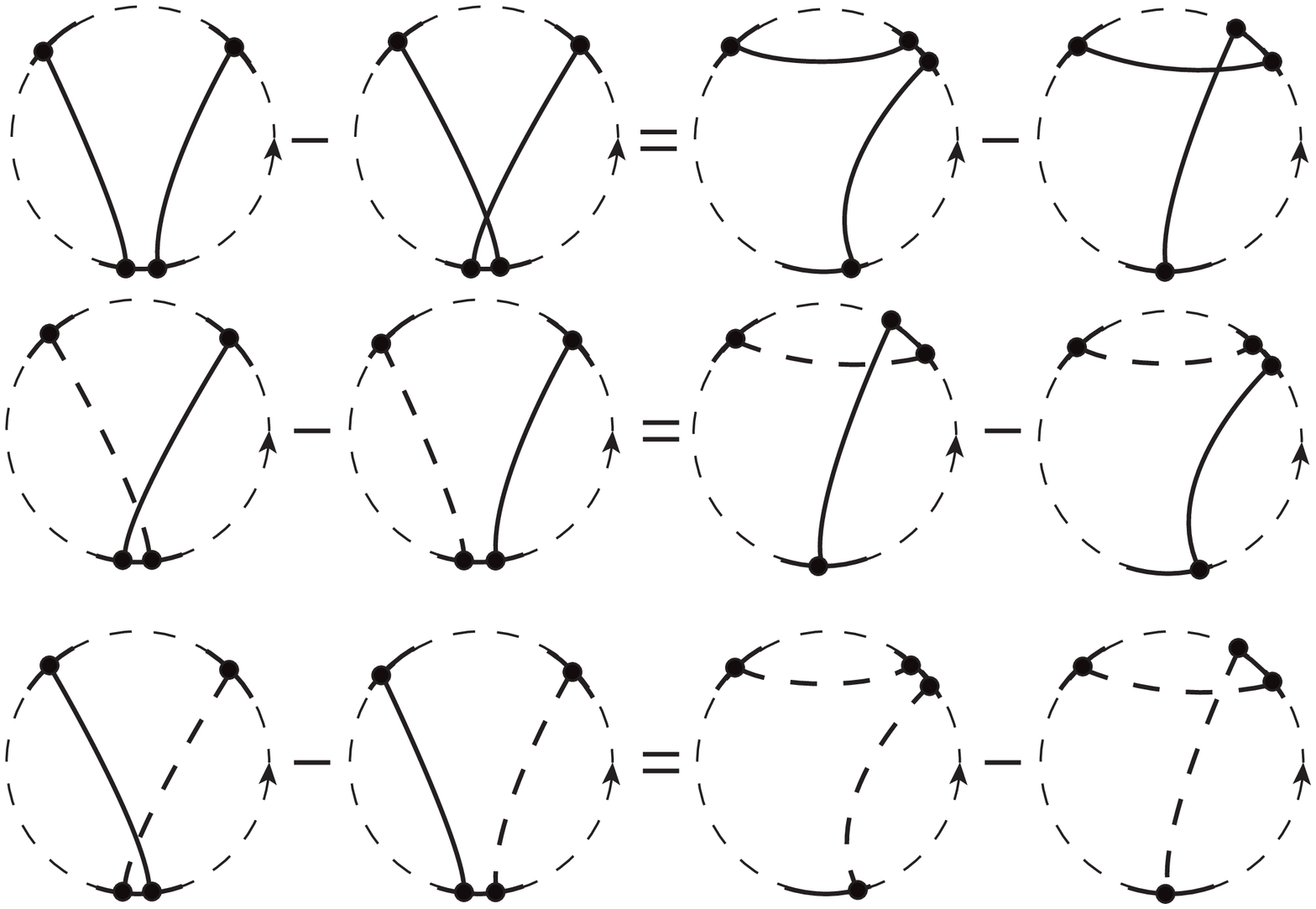}
  \caption{$4$T-relations for framed chord diagrams}\label{4trel}
 \end{figure}

 \begin{remark}
The pictures in Fig.~\ref{4trel} should be understood as follows. It
is assumed that the endpoints of other chords can lie only in the
dashed parts of the core circle and the combinatorial structure of
chords not depicted in the pictures, is the same for all the four
diagrams constituting the relation.
 \end{remark}

We can consider chord diagrams only with framings $0$ and the
corresponding $4$T-relation. As a result we obtain the submodule
$\mathcal{M}$ of $\mathcal{M}^f$, see~\cite{Kar}.

Considering the field $\mathbb{R}$ instead of $\mathbb{Z}$ we can
obtain a commutative cocommutative Hopf algebra $\mathcal{A}$ of
chord diagrams~\cite{ChDM}. The space $\mathcal{A}$ is endowed with
a natural product $m\colon
\mathcal{A}\otimes\mathcal{A}\to\mathcal{A}$, a natural coproduct
$\Delta\colon\mathcal{A}\to\mathcal{A}\otimes\mathcal{A}$, the unit
$e\colon\mathbb{R}\to\mathcal{A}$, the counit
$\epsilon\colon\mathcal{A}\to\mathbb{R}$ and the antipode
$\mathcal{S}\colon\mathcal{A}\to\mathcal{A}$. The multiplication is
given by gluing two diagrams, i.e.\ a connected sum, and the
coproduct is given by summing up the tensor products of pairs of
chord diagrams formed by a decomposition of the set of chords into
two complimentary subsets. Analogously, the space $\mathcal{A}^f$ of
framed chord diagrams is endowed with the comultiplication
transforming it to the coassociative cocommutative
coalgebra~\cite{Kar,Lan}.

%%%%%%%%%%%%%%%%%%%%%%%%%%%%%%%%%%%%%%%%%%%%%%%%%%%%%%%%%%%%%%%%%%%%%%
\section{Double chord diagrams}\label{sec:dchd}
%%%%%%%%%%%%%%%%%%%%%%%%%%%%%%%%%%%%%%%%%%%%%%%%%%%%%%%%%%%%%%%%%%%%%%

In this section we construct chord diagrams on two circles.
Throughout this section all chords have the framing $1$.

 \subsection{Basic definitions}

 \begin{definition}
A {\em double chord diagram} is a cubic graph consisting of two
oriented disjoint circles (the {\em core circles}) and several
non-oriented edges ({\em chords}) connecting points on the core
circles in such a way that every point on a core circle is incident
to at most one chord.
 \end{definition}

 \begin{remark}
Double chord diagrams are considered up to isomorphisms of graphs
preserving the orientations of two core circles.
 \end{remark}

Define the $\mathbb{Z}$-module $M_2$ as the set of finite
$\mathbb{Z}$-linear combinations of double chord diagrams.

On the set of all double chord diagrams we can define relations
analogous to the $4$T-relation on chord diagrams without chords with
framing $1$. The difference in the definition of these moves is the
following. In the case of double chord diagrams three pieces
containing the endpoints of two singled chords may lie in the two
core circles. We refer to these relations, see Fig.~\ref{4treld}, as
to  {\em $4$T-relations}.

 \begin{figure}
  \centering\includegraphics[width=300pt]{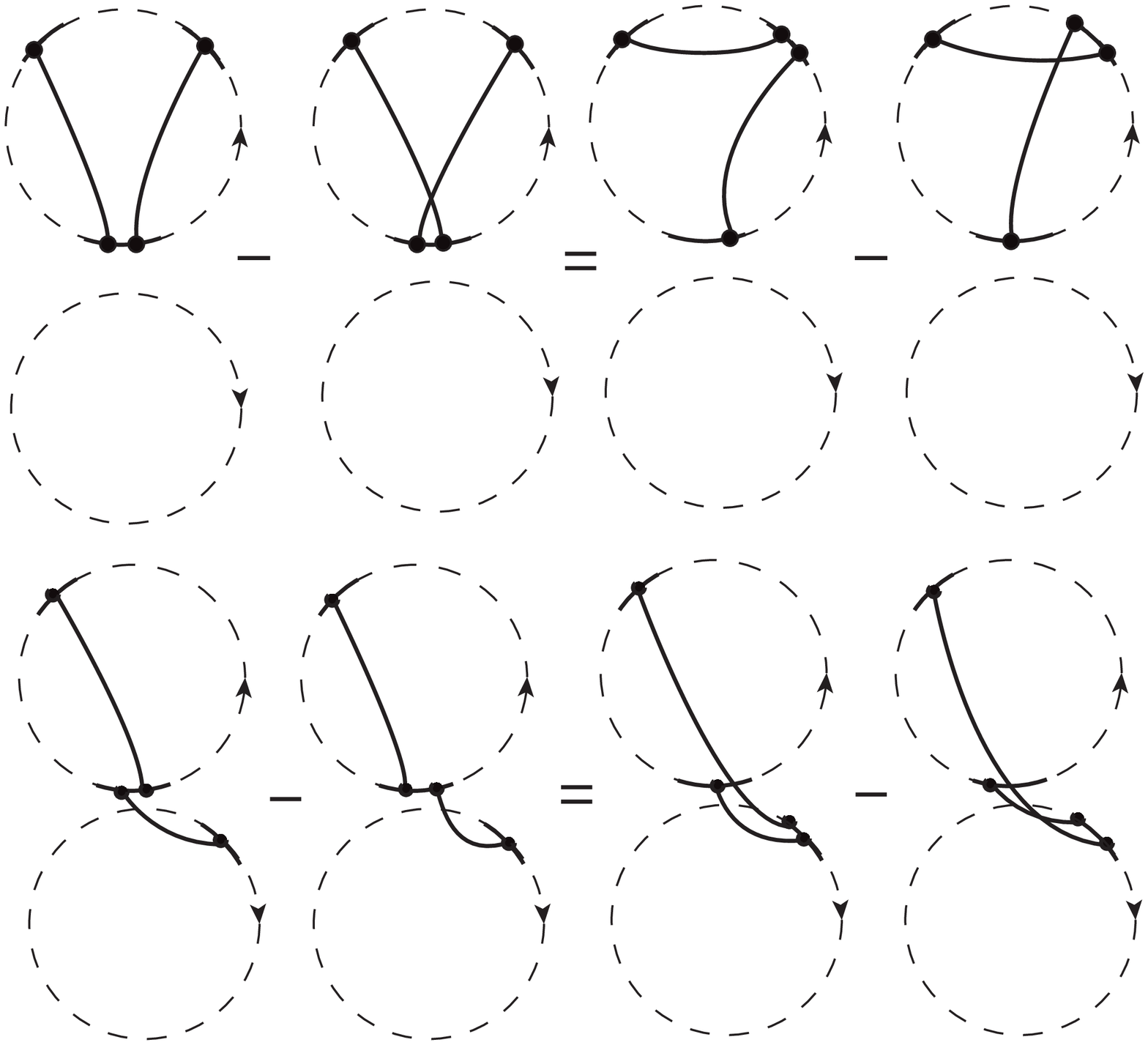}
  \caption{$4$T-relations for double chord diagrams}\label{4treld}
 \end{figure}

 \begin{definition}
The module $\mathcal{M}_2$ of double chord diagrams is the quotient
module of $M_2$ modulo the $4$T-relations.
 \end{definition}

 \subsection{A weight system}\label{subswc:ws}

We recall that a linear map from the space of framed chord diagrams
satisfying $4$T-relations is called a {\em weight system}.

 \begin{definition}
We call any linear map from the space of double chord diagrams
satisfying $4$T-relations a {\em weight system} on the set of double
chord diagrams.
 \end{definition}

Let us construct one example of a weight system on  $\mathcal{M}_2$.

Let us first define a {\em surgery along the set of chords} of a
double chord diagram, as it was done for chord diagrams and framed
chord diagrams in
see~\cite{BNG,CL,I-ko:f4vg,I-ko:eq,Lan,Mel,Mo,Sob,Tr}.

Let $\mathcal{D}$ be a double chord diagram. For every chord
belonging to one core circle we draw a parallel chord near it and
remove the small arcs of the core circle between adjacent ends of
the chords. For every chord with endpoints on two core circles we
replace it with two chords in such a way that after removing the
small arcs of the core circles between adjacent ends of the chords
the orientations of the core circles are coherent, see
Fig.~\ref{surgd}.

By a small perturbation, the picture in $\mathbb{R}^{2}$ is
transformed into a one-manifold $N(\mathcal{D})$ in
$\mathbb{R}^{3}$. Let $\beta_{\mathcal{D}}$ be the number of
connected components of $N(\mathcal{D})$.

 \begin{figure}
  \centering\includegraphics[width=200pt]{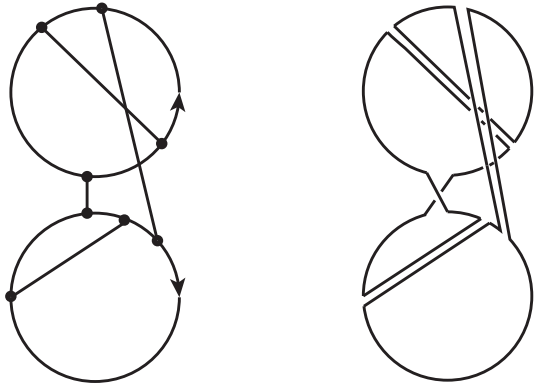}
  \caption{The surgery}\label{surgd}
 \end{figure}

The proof of the following theorem is analogous to the proof of the
corresponding theorem for framed chord diagrams, see~\cite{Mel}.

 \begin{theorem}\label{th:ws}
Let $\mathcal{D}$ be a double chord diagram. Then
$\beta_{\mathcal{D}}$ is invariant under $2$T-relations{\em,} see
Fig.~{\em\ref{2Treld}}.
 \end{theorem}

 \begin{figure}
  \centering\includegraphics[width=300pt]{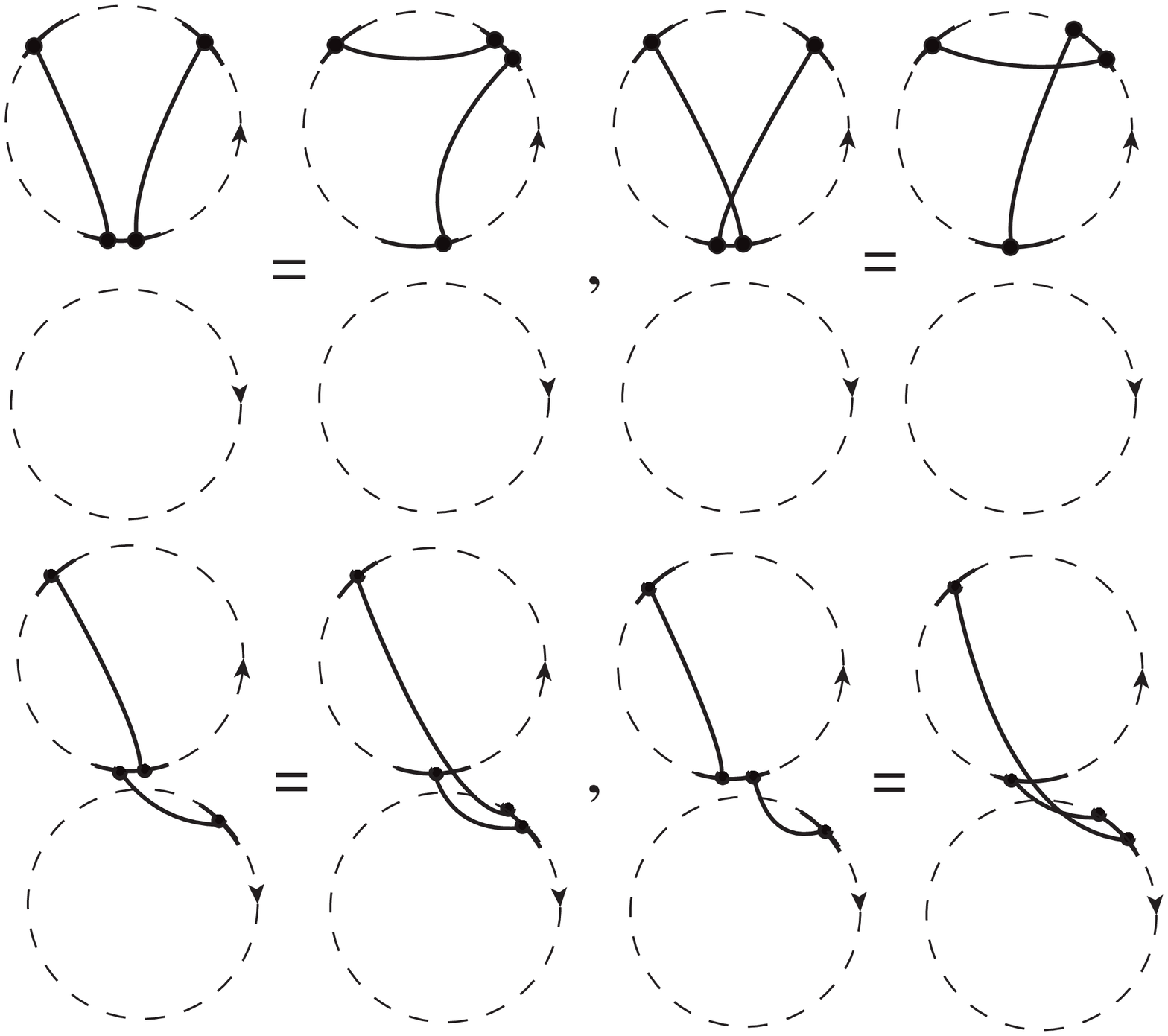}
  \caption{$2$T-relations for double chord diagrams}\label{2Treld}
 \end{figure}

 \begin{corollary}
The number $\beta_{\mathcal{D}}$ is invariant under $4$T-relations.
 \end{corollary}

Define a map $w\colon\mathcal{M}_2\to\mathbb{Z}$ by putting
 $$
w(\alpha_1\mathcal{D}_1+\ldots+\alpha_k\mathcal{D}_k)=\alpha_1\beta_{\mathcal{D}_1}+\ldots+\alpha_k\beta_{\mathcal{D}_k},
 $$
where $\alpha_i\in\mathbb{Z}$. From Theorem~\ref{th:ws} we see that
$w$ is a weight system on $\mathcal{M}_2$.

%%%%%%%%%%%%%%%%%%%%%%%%%%%%%%%%%%%%%%%%%%%%%%%%%%%%%%%%%%%%%%%%%%%%%%
 \section{A map from the module $\mathcal{M}^f$ to the module
 $\mathcal{M}_2$}\label{sec:map}
%%%%%%%%%%%%%%%%%%%%%%%%%%%%%%%%%%%%%%%%%%%%%%%%%%%%%%%%%%%%%%%%%%%%%%

It turns out that there is a well-defined map $\psi$ from the module
$\mathcal{M}^f$ to the module $\mathcal{M}_2$.

%Before constructing this map we need one more definition.

% \begin{definition}
%Let $D_1$ and $D_2$ be two (non-framed) chord diagrams. Define the
%{\em $n$-connected sum}, $n\in \mathbb{N}\cup\{0\}$, of $D_1$ and
%$D_2$ by adding $n$ chords having endpoints on both the core circles
%of $D_1$ and $D_2$, and reversing the orientation of the core circle
%of $D_2$. If we are not interested in the number $n$ of added
%chords, we call this operation just a {\em connected sum}.
% \end{definition}

%Let us first define $\psi$ on a framed chord diagram $D$
%consisting of $n$ chords $d_i=\{u_i,v_i\}$, $i=1,\dots,n$, where $k$
%chords $d_1,\dots,d_k$ have framing $0$. Denote by
%$D^{i_1,\dots,i_p}$ the framed chord diagram obtained from $D$ by
%deleting $p$ chords with numbers $i_1,\dots,i_p$. The result of
%$\psi$, $\psi(D)$, is the sum of $2^{n}$ summands, each
%containing of $n$ chords.

%Let $D_1$ with chords $d^1_i=\{u^1_i,v^1_i\}$ and $D_2$ with chords
%$d^2_i=\{u^2_i,v^2_i\}$ be two copies of $D$. Then each summand of
%$\psi(D)$ is the $(n-k)$-connected sum obtained from two framed
%chord diagrams $D_1^{i_1,\dots,i_q,k+1,\dots,n}$ and
%$D_2^{1,\dots,\widehat{i_1},\dots,\widehat{i_q},\dots,k,k+1,\dots,n}$,
%where $1\leqslant i_1\leqslant i_2\leqslant\ldots\leqslant
%i_q\leqslant k$ and the hat $\widehat{\cdot}$ means that the
%corresponding index is omitted, by adding $n-k$ chords connecting
%either endpoints $u^1_j$ and $v^2_j$ or $v^1_j$ and $u^2_j$.

Let  us first define $\psi$ on a framed chord diagram $D$ with a
core circle $C$ and $n$ chords. Construct double chord diagrams with
the core circles $C_1$ and $C_2$ to be $C$ as follows. For each
chord $d$ of $D$ consider the two positions of it in $C_1$ and
$C_2$. Namely, if $d$ has framing $0$, then it can have its both
endpoints either on $C_1$ or $C_2$. If $d$ has framing $1$, then its
endpoints lie on different core circles. This leads to $2^n$
possible choices. We define $\psi(D)$ to be the sum of these $2^n$
summands but the orientation of $C_2$ is reversed.

Then we extend the map $\psi$ by linearity.

 \begin{example}
Let $D$ be a framed chord diagram depicted in Fig.~\ref{exch}. Then
 \begin{align*}
\psi(D)&=\psi\left(\exam\right)=\exone+\extwo+\exthree+\exfour\\
&+\exfive+\exsix+\exseven+\exeight.
 \end{align*}
 \end{example}

 \begin{theorem}
The map $\psi\colon\mathcal{M}^f\to\mathcal{M}_2$ is well defined.
 \end{theorem}

 \begin{proof}
We just have to check that $\psi$ preserves the $4$T-relations. We
consider only the third relation from Fig.~\ref{4trel}. We have
 \begin{gather*}
\psi\left(\fl\right)=\flo+\flt+\flth+\flf,\\
\psi\left(\slef\right)=\slo+\slt+\slth+\slf,\\
\psi\left(\fr\right)=\fro+\frt+\frth+\frf,\\
\psi\left(\sref\right)=\sro+\srt+\srth+\srf.
 \end{gather*}

As a result we have
 \begin{multline*}
\psi\left(\fl-\slef\right)=\flt-\slt+\flth-\slth\\
=\frt-\srt+\frth-\srth=\psi\left(\fr-\sref\right).
 \end{multline*}

The other relations are checked analogously.
 \end{proof}

%%%%%%%%%%%%%%%%%%%%%%%%%%%%%%%%%%%%%%%%%%%%%%%%%%%%%%%%%%%%%%%%%%%%%%
 \section{An application of $\psi$}\label{sec:appl}
%%%%%%%%%%%%%%%%%%%%%%%%%%%%%%%%%%%%%%%%%%%%%%%%%%%%%%%%%%%%%%%%%%%%%%

It is well-known that the connection sum of chord diagrams is well
defined~\cite{BN}. In this section, by using $\psi$ we show that a
connected sum of framed chord diagrams is not a well-defined
operation in $\mathcal{M}^f$. Consider two framed chord diagrams
$D_1$ and $D_2$ depicted in Fig.~\ref{twofchd}.

 \begin{figure}
  \centering\includegraphics[width=200pt]{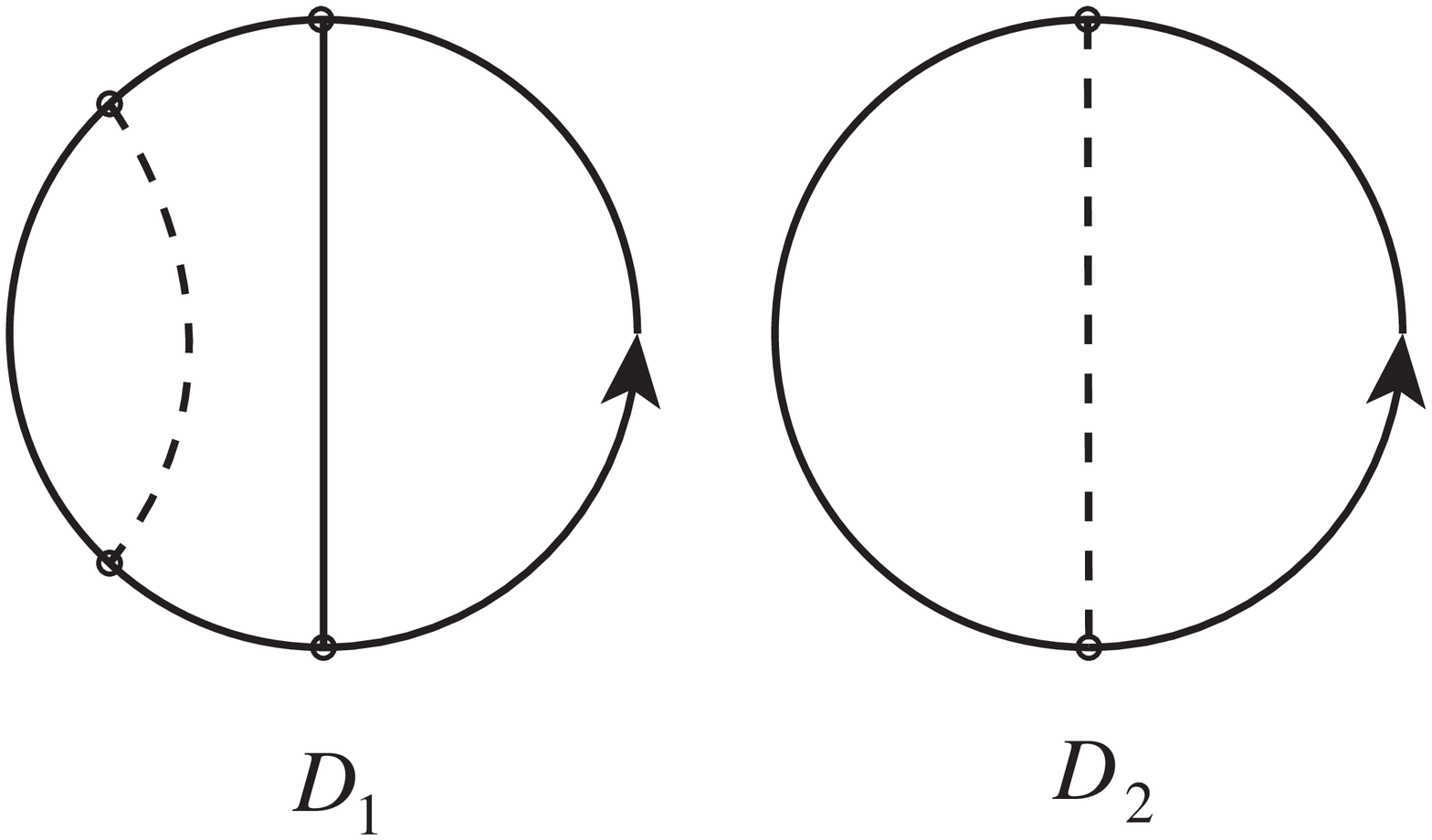}
  \caption{Two framed chord diagrams}\label{twofchd}
 \end{figure}

Choosing points on the chord diagrams in different ways we can
obtain the following two connected sums $D$ and $D'$, see
Fig.~\ref{twoconsum}.

 \begin{figure}
  \centering\includegraphics[width=200pt]{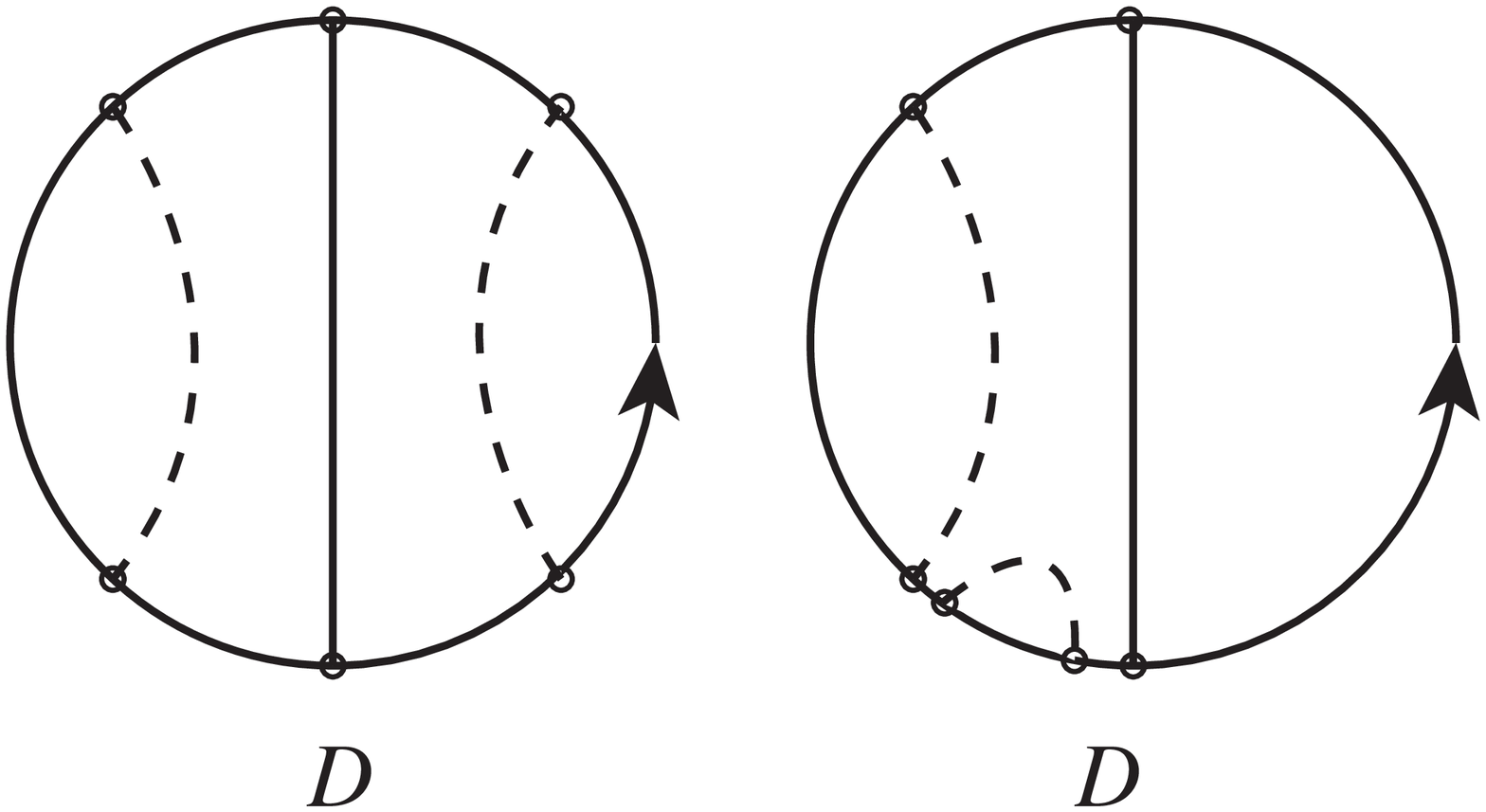}
  \caption{Two connected sums}\label{twoconsum}
 \end{figure}

We have
 \begin{gather*}
\psi(D)=4\twocsone+4\twocstwo=8\twocsone,\\
\psi(D')=4\twocsthree+4\twocsfour=8\twocsthree.
 \end{gather*}
It is not difficult to show that the number of connected components
of $N\left(\twocsone\right)$ is equal to $1$, but the number of
connected components of $N\left(\twocsthree\right)$ equals $3$.
Therefore, $w(\psi(D))=8$ and $w(\psi(D'))=24$ and the elements
$8\twocsone$ and $8\twocsthree$ do not coincide in $\mathcal{M}_2$,
so do the elements $D$ and $D'$ in $\mathcal{M}^f$.

%%%%%%%%%%%%%%%%%%%%%%%%%%%%%%%%%%%%%%%%%%%%%%%%%%%%%%%%%%%%%%%%%%%%%%
 \section{Linear diagrams}\label{sec:ld}
%%%%%%%%%%%%%%%%%%%%%%%%%%%%%%%%%%%%%%%%%%%%%%%%%%%%%%%%%%%%%%%%%%%%%%

Besides framed chord diagrams and double chord diagrams, we can
consider linear diagrams and double linear diagrams.

 \begin{definition}
A {\em linear diagram} is an oriented line with a finite number of
arcs having their endpoints on this line. A linear diagram is {\em
framed} if a map (a {\em framing}) from the set of arcs to
$\mathbb{Z}/2\mathbb{Z}$ is given, i.e.\ every arc is endowed with
$0$ or $1$.
 \end{definition}

 \begin{remark}
We consider all framed linear diagrams up to orientation and framing
preserving isomorphisms of graphs taking one line to the other one.
Arcs having framing $0$ are solid arcs, and those having framing
$1$, are dashed ones, see Fig.~\ref{ex_fld}.

Having a framed linear diagram $G$, we can construct the framed
chord diagram, the closure $\mathrm{Cl}(B)$ of $G$, by ``closing''
the line. It is not difficult to see that this operation (the map
from the set of linear framed chord diagram to the set of framed
chord diagram) is well defined.
 \end{remark}

 \begin{figure}
  \centering\includegraphics[width=200pt]{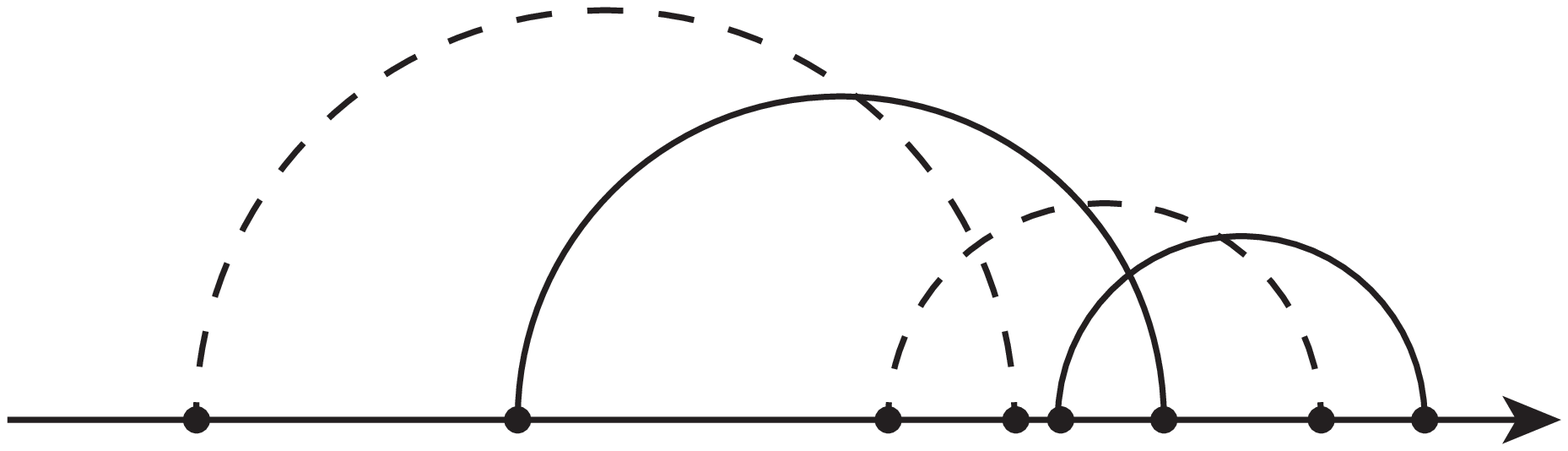}
  \caption{A framed linear diagram with two chords of framing $0$ and two chords of framing $1$}\label{ex_fld}
 \end{figure}

Let $L^f$ be the free $\mathbb{Z}$-module generated by all framed
linear diagrams and $\mathcal{L}^f$ be the quotient module of $L^f$
modulo the relations shown in Fig.~\ref{4treld}. We refer to these
relations as to {\em linear $4$T-relations}.

 \begin{figure}
  \centering\includegraphics[width=300pt]{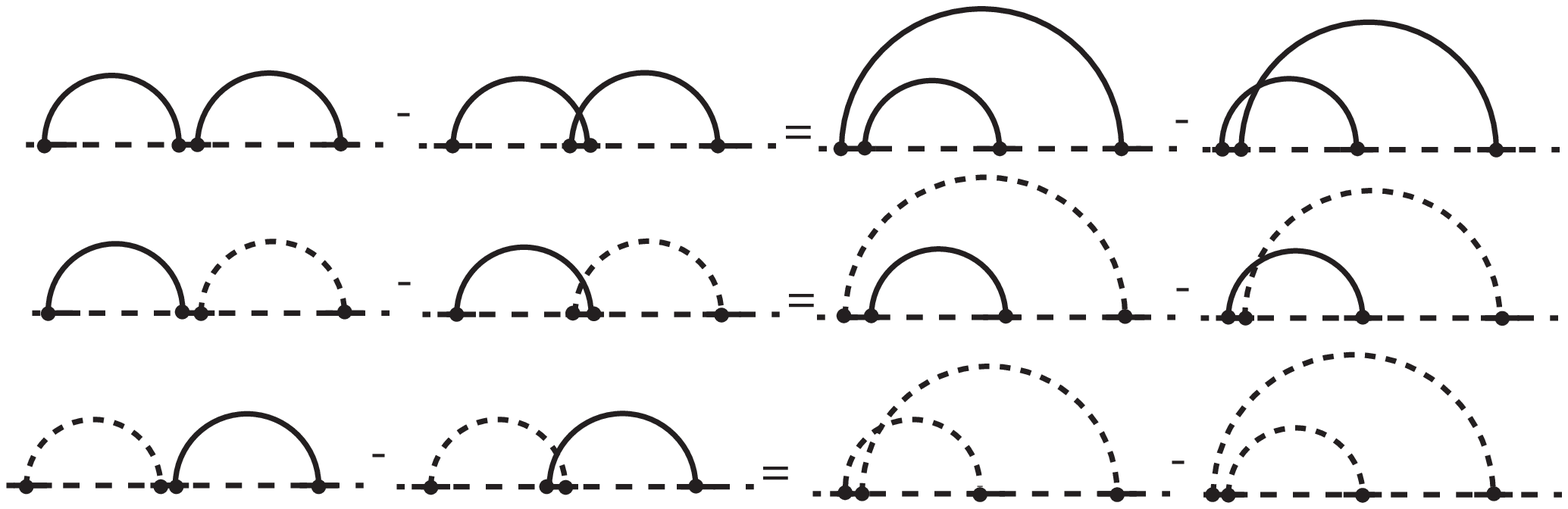}
  \caption{Linear $4$T-relations for framed linear diagrams}\label{4trel_ld}
 \end{figure}

 \begin{remark}
In Fig.~\ref{4trel_ld}, the lines on the LHS and RHS of each
equality are assumed to be oriented accordingly.
 \end{remark}

Analogously, we can consider chord diagrams only with framings $0$
and the corresponding linear $4$T-relations. As a result we obtain
the submodule $\mathcal{L}$ of $\mathcal{L}^f$. The module
$\mathcal{L}$ can be endowed with the structure of a commutative
cocommutative Hopf algebra, where the a product
$\mathcal{L}\otimes\mathcal{L}\to\mathcal{L}$ is given by gluing two
lines according to the orientation.

 \begin{definition}
A {\em double linear diagram} is a union of two disjoint framed
linear diagrams and arcs with endpoints on both the oriented lines.
 \end{definition}

 \begin{remark}
Double linear diagrams are considered up to isomorphisms preserving
the orientations of lines.
 \end{remark}

Let $L_2$ be the free $\mathbb{Z}$-module of double linear diagrams
and let $\mathcal{L}_2$ be the quotient module of $L_2$ modulo the
relations ({\em $4$T-relations}) shown in Fig.~\ref{4treldld}.

 \begin{figure}
  \centering\includegraphics[width=300pt]{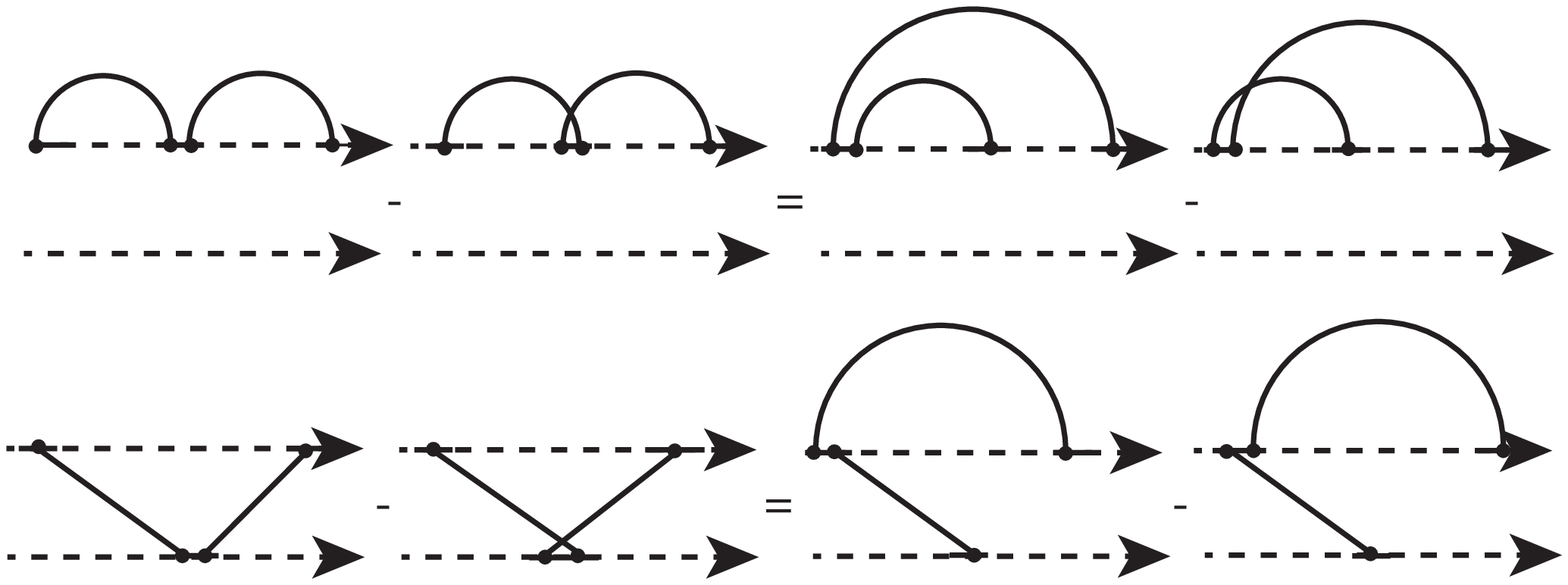}
  \caption{Linear $4$T-relations for double linear diagrams}\label{4treldld}
 \end{figure}

 \begin{definition}
Any linear map from the space of double linear diagrams satisfying
linear $4$T-relations is called a {\em weight system} on the set of
double linear diagrams.
 \end{definition}

Let $\mathcal{G}$ be a double chord diagram. Analogously to
Sec.~\ref{subswc:ws} we can define the {\em surgery} for
$\mathcal{G}$ and obtain a one-manifold $N(\mathcal{G})$, see
Fig.~\ref{surgdld}. $N(\mathcal{G})$ consists of two lines and a
collection of circles.

 \begin{figure}
  \centering\includegraphics[width=150pt]{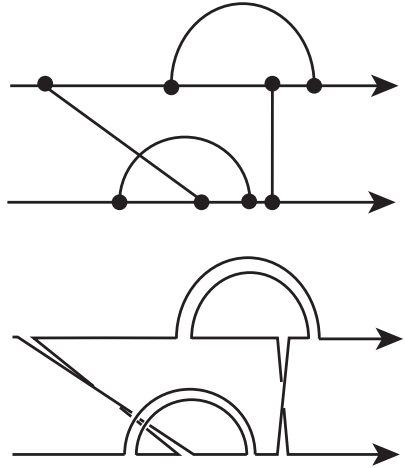}
  \caption{The surgery}\label{surgdld}
 \end{figure}

 \begin{theorem}\label{th:wsdld}
The number $\beta_{\mathcal{G}}$ of connected components of
$N(\mathcal{G})$ is invariant under linear $2$T-relations{\em,} see
Fig.~{\em\ref{2Treldld}}.
 \end{theorem}

 \begin{figure}
  \centering\includegraphics[width=300pt]{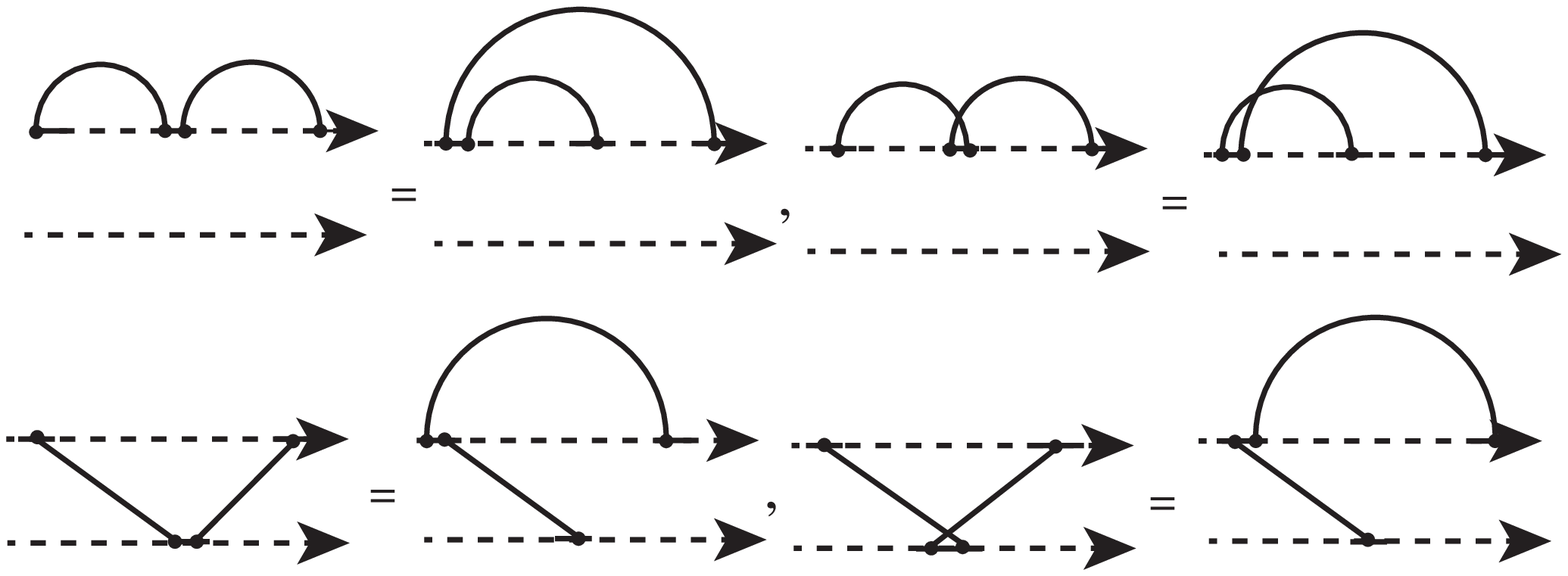}
  \caption{Linear $2$T-relations for double linear diagrams}\label{2Treldld}
 \end{figure}

 \begin{corollary}
The number $\beta_{\mathcal{G}}$ is invariant under linear
$4$T-relations.
 \end{corollary}

Define a map $w_l\colon\mathcal{L}_2\to\mathbb{Z}$ by putting
 $$
w_l(\alpha_1\mathcal{G}_1+\ldots+\alpha_k\mathcal{G}_k)=\alpha_1\beta_{\mathcal{G}_1}+\ldots+\alpha_k\beta{\mathcal{G}_k},
 $$
where $\alpha_i\in\mathbb{Z}$. Using Theorem~\ref{th:wsdld} we get
that the map $w_l$ is a weight system on $\mathcal{L}_2$.

Let us construct a map $\psi_l\colon\mathcal{L}^f\to\mathcal{L}_2$
in the same way as it was done in Sec.~\ref{sec:map}. For example,
 \begin{align*}
\psi_l\left(\exfld\right)&=\edldo+\edldt\\
&+\edldth+\edldf\\
&+\edldfv+\edlds\\
&+\edldse+\edlde.
 \end{align*}

% \begin{figure}
%  \centering\includegraphics[width=150pt]{ex_fld.eps}
%  \caption{Framed linear diagrams}\label{exfld}
% \end{figure}

 \begin{theorem}
The map $\psi_l\colon\mathcal{L}^f\to\mathcal{L}_2$ is well defined.
 \end{theorem}

Having two linear framed diagrams $G_1$ and $G_2$ we can define
their connected sum by gluing one line to the other one. For
example, if $G_1=\lfdone$ and $G_2=\lfdtwo$, then $G_1\#
G_2=\fconsum$ and $G_2\# G_1=\sconsum$.

Opposite to the case of framed chord diagrams, we have the following
theorem.

 \begin{theorem}
For any two linear framed diagrams we have
 $$
w_l(\psi_l(G_1\# G_2))=w_l(\psi_l(G_2\# G_1)).
 $$
 \end{theorem}

 \begin{proof}
Let us mark out the lines in each summand of $\psi_l(G_1)$ and
$\psi_l(G_2)$. Then each summand of $\psi_l(G_1\# G_2)$ and
$\psi_l(G_2\# G_1)$ is the connected sum of a summand from
$\psi_l(G_1)$ and a summand from $\psi_l(G_2)$, where the first
(second) line is connected with the first (second) one. The validity
of the theorem follows from the following fact.

Let $\mathcal{G}_1$  and $\mathcal{G}_2$ be two double linear
diagrams with marked lines. Then
$\beta_{\mathcal{G}_1\#\mathcal{G}_2}$ is equal to
$\beta_{\mathcal{G}_1}+\beta_{\mathcal{G}_2}-1$ or
$\beta_{\mathcal{G}_1}+\beta_{\mathcal{G}_2}-2$, where $1$ and $2$
depend on the initial diagram $\mathcal{G}_1$ and $\mathcal{G}_2$,
but do not depend on the connected sums. This statement can be
easily proved by analyzing possible connections of non-compact
components of $N(\mathcal{G}_1)$ and $N(\mathcal{G}_2)$.
 \end{proof}

%Consider two framed linear diagrams $G_1$ and $G_2$ depicted in Fig.~\ref{twofld}. We have $G_1\# G_2=\fconsum$ and $G_2\# G_1=\sconsum$.

% \begin{figure}
%  \centering\includegraphics[width=200pt]{2fld.eps}
%  \caption{Two framed linear diagrams}\label{twofld}
% \end{figure}

%Further, we get
% \begin{gather*}
%\psi_l(G_1\# G_2)=4\fconsumo+4\fconsumt,\\
%\quad\psi_l(G_2\# G_1)=4\sconsumo+4\sconsumt=8\sconsumo,\\
%w_l(G_1\# G_2)=32,\quad w_l(G_2\# G_1)=32.
% \end{gather*}
%Therefore, the elements $G_1\# G_2$ and $G_2\# G_1$ do not coincide
%in $\mathcal{L}^f$.

 \section*{Acknowledgments}

The authors are grateful to I.\,M.~Nikonov for his interest to this
work.

 \end {document}